\documentclass[11pt]{amsart}
\usepackage{amsfonts, amsmath, amssymb, amscd, amsthm, graphicx, setspace, enumitem, color}
\usepackage{hyperref}
\usepackage[all]{xy}
\newcommand{\ol}[1]{\overline{{#1}}}

\newcommand{\ra}{\rightarrow}

\newcommand{\gde}{{\underline{\underline{\textrm{gd}}}\ }}

\newcommand{\E}{{\underline{\underline{E}}}}
\newcommand{\F}{{\mathcal{F}}}

\def\coloneqq{\mathrel{\mathop\mathchar"303A}\mkern-1.2mu=}

\newtheorem{thm}{Theorem}[section]
\newtheorem{cor}[thm]{Corollary}
\newtheorem{lem}[thm]{Lemma}
\newtheorem{prop}[thm]{Proposition}

\newtheorem{conv}[thm]{Convention}
\theoremstyle{definition}
\newtheorem{defn}[thm]{Definition}

\newtheorem{nt}[thm]{Notation}

\theoremstyle{remark}
\newtheorem{rem}[thm]{Remark}
\newtheorem{ex}[thm]{Example}

\newcommand{\gen}[1]{\left\langle#1\right\rangle}

\newcommand{\gp}[2]{\gen{#1\mid #2}}

\DeclareMathOperator{\Comm}{Comm}
\DeclareMathOperator{\stab}{Stab}

\DeclareMathOperator{\uE}{\underline{E}}

\newcommand{\Z }{\mathbb Z}

\def\coloneq{\mathrel{\mathop\mathchar"303A}\mkern-1.2mu=}

\begin{document}

\title[Bredon homology of Artin groups of dihedral type]{Bredon homology of Artin groups of dihedral type}

\address{Departamento de  Álgebra, Geometría y Topología, Universidad Complutense de
Madrid and Instituto de Ciencias Matemáticas, CSIC-UAM-UC3M-UCM}
\author{Yago Antol\'{i}n}
\email[Yago Antol\'{i}n]{yago.anpi@gmail.com}

\address{Universidad de Sevilla,
Departmento de Geometr\'{i}a y Topolog\'{i}a,
c/ Tarfia, s/n 41080, Seville (Spain)}
\author{Ram\'on Flores}
\email[Ram\'on Flores]{ramonjflores@us.es}

\footnotetext{


The second author was supported by grant PID2020-117971GB-C21 of the Spanish Ministery of Science and Innovation, and grant FQM-213 of the Junta de Andaluc\'{\i}a.}

\thanks{}

\begin{abstract}

For Artin groups of dihedral type, we compute the Bredon homology groups of the classifying space for the family of virtually cyclic subgroups with coefficients in the {\it K}-theory of a group ring.

\end{abstract}


\subjclass[2010]{}

\maketitle


\section{Introduction}

This paper is part of a program aimed to understand equivariant and $K$-theoretic invariants for Artin groups.
These groups are ubiquitious in geometric group theory, as they comprehend many different subfamilies of groups (braid groups, right-angled, spherical type, extra-large, free groups...), provide examples and counterexamples for many interesting phenomena, and at the same time they are not well understood from a global point of view: it is not even known if every Artin group is torsion-free. See \cite{Par09} for a good survey. In this context, Azzali  \emph{et al.} have recently computed explicitly both sides of Baum-Connes for pure braid groups \cite{ABGRW}; in joint work of the second author with J. Gonz\'alez-Meneses \cite{FlGo18} it is computed the minimal dimension of a model of $\E G$ for braid groups; J. and virtually-cyclic dimensions of mapping class groups (which include in particular certain Artin groups) have been recently investigated by Aramayona \emph{et al.}  (\cite{AJT18}, \cite{ArMa}) and by Petrosyan and Nucinkis \cite{NuPe18}.


In the present paper we consider the case of Artin groups of dihedral type, denoted by $A_n$, which are the groups defined by
$$A_n=\gp{a,b}{\mathrm{prod}(a,b;n)=\mathrm{prod}(b,a;n)}.$$
Here  $\mathrm{prod}(x,y;n)$ denotes the word of length $n$ that alternates $x$ and $y$ and starts with $x$.
 These groups   are one-relator and torsion-free, they  have small geometric dimension and its group-theoretic structure is quite well understood. These features makes them strongly appropriate for computations, and in fact have been recently studied from different angles, as growth series \cite{MaMa06}, systems of equations \cite{CHR20} or geodesics \cite{Wal09}.

We compute here Bredon homology groups of the classifying space of these groups with respect to the family of virtually cyclic groups. Bredon homology was first described by Glen Bredon in the sixties \cite{Bre67}, and it is a $G$-equivariant homology theory that takes into account the action over the target space of the subgroups of $G$ that belong to a predefined family. Since their appearance, Bredon homology groups have played a prominent role in Homotopy Theory and Group Theory, particularly in relation with finiteness conditions, dimension theory for groups and classifying spaces, and in the framework of the Isomorphism Conjectures. In fact, our choice of coefficients in the $K$-theory of a ring $G$ has been done with an eye in possible applications of the computations to the left-hand side of Farrell-Jones Conjecture (see below), although we believe that the methods of the present paper may be useful in a more general context, switching the $K$-theoretic coefficients to a general coefficient module.

Our main technical result is the following:


\noindent\textbf{Theorem \ref{Thm:Bredon}.} Let $A_n$ be an Artin group of dihedral type and $R$ a ring, $R$ a ring, $K_q(R[-])$ the covariant module over the orbit category $O_{\mathcal F}(G)$ that sends every (left) coset  $G/H$ to the $K$-theory group $K_q(R[H])$. Then we have the following:
\begin{enumerate}
  \setlength{\itemindent}{-2em}
\item $H_i^{vc}(\E A_n,K_q (R[-]))=\{0\}$ for $i\geq 4$.
\item $H_3^{vc}(\E A_n,K_q (R[-]))=\begin{cases}\bigoplus_{[H]\neq [Z(A_n)]} K_q(R) & \text{$n$ odd}\\
\ker g_2^2 & \text{$n$ even.}
\end{cases}$

\item $H_2^{vc}(\E A_n,K_q (R[-]))= \ker g_2^1$.

\item $H_1^{vc}(\E A_n,K_q (R[-]))= \emph{coker } g_2^1 = \begin{cases}(\bigoplus_{[H]\neq [Z(A_n)]}N_q^{[H]})\oplus T_1(K_q(R))\oplus T_2(K_q(R)) & \text{$n$ odd}\\
\bigoplus_{[H]} N_q^{[H]}
  & \text{$n$ even.}
\end{cases}$%
\item $H_0^{vc}(\E A_n,K_q (R[-]))= \emph{ coker } g_2^0= \begin{cases}(\bigoplus_{[H]\neq [Z(A_n)]} N_q^{[H]})\oplus K_q(R)\oplus\overline{C}(K_q(R)) & \text{$n$ odd}\\
(\bigoplus_{[H]} N_q^{[H]})\oplus K_q(R)
  & \text{$n$ even.}
  \end{cases}$%
\end{enumerate}


Here $g^j_i$ stands for a homomorphism in the Degrijse-Petrosyan exact sequence  {\cite[Section 7]{DP}}, $[H]$ for the commensurability class of a {non-trivial} cyclic subgroup $H$, and $N_q^{[H]}$, $T_i(K_q(R))$ and $\overline{C}(K_q(R))$ for groups that depend on $H$ and the Bass-Heller-Swan decomposition of $K_q(R[\Z])$. See Section \ref{Sect:BredonArtin} for details.



The main tool used in the proof of Theorem \ref{Thm:Bredon} is an exact sequence in Bredon homology \cite{DP}, which is in turn the Mayer-Vietoris sequence associated to the push-out that defines the L\"{u}ck-Weiermann model for $\E G$ (\cite{LW}, see also Section \ref{Sect:Prelim} below). The knowledge about the group-theoretic structure of the groups includes a complete understanding of the commensurators, which is crucial in the computations. Using the theorem, we are able to describe with precision the Bredon homology of $A_n$ with coefficients in the $K$-theory of several rings, both regular and non-regular (see Section \ref{Sect:concrete}). 

Next we describe the implications of our work in relation with the Farrell-Jones Conjecture. Recall that given a group $G$, a ring $R$ and $n\in\mathbb{Z}$, Farrell-Jones stated in \cite{FaJo93} the existence of an assembly map:

$$H_n^G(\E G,\mathbf{K} (R))\rightarrow K_n(RG),$$
where $\E G$ is the classifying space of $G$ with respect to the family of virtually cyclic groups,
$H_*^G(-,\mathbf{K} (R))$ is the $G$-homology theory defined in Section 1 of \cite{FaJo93} and $K_n(RG)$ stands for the $n$-th group of algebraic $K$-theory of the group ring $RG$.
The Farrell-Jones conjecture predicts that the assembly map is an isomorphism.
The conjecture has been verified for a big family of groups (see \cite{LR05} for an excellent survey), and no counterexample has been found so far.
The philosophy in this context is, for a group for which the assembly map is known to be an isomorphism, to perform computations in the topological side in order to extract information about the algebraic $K$-theory of the group ring. It is remarkable that the latter are difficult to compute, and at the same time encode fundamental invariants of manifolds, including obstructions to the existence of cobordisms and information about groups of pseudoisotopies. Explicit calculations in this context can be found in \cite{BuSa16}, \cite{DQR11}, \cite{KLL21} or \cite{SaVe18}, for example.

The left-hand side of the conjecture can be approached by means of a $G$-equivariant version of the Atiyah-Hirzebruch spectral sequence, which  converges to the Farrell-Jones $K$-homology, and whose $E_2$-page is  the Bredon homology of the classifying space $\E G$ with coefficients in the $K$-theory of the group rings of the virtually cyclic subgroups of $G$; example of such calculations in the Farrell-Jones framework can be found in \cite{BJV14} \cite{LuRo14}. In this setting, Theorem \ref{Thm:Bredon} and the examples of Section \ref{Sect:concrete} can be interpreted as an explicit computation of such $E_2$-page, in the case of Artin groups of dihedral type. We remark that in the case of $R$ regular, the left-hand side of the conjecture for these Artin groups can be deduced from \cite[Lemma 16.12]{Luc21}, using previously a splitting result of L\"{u}ck-Steimle \cite{LS16} (see the end of Section \ref{Sect:concrete} for details), so our results provide new information in this context for a non-regular $R$; in this sense, we expect that Examples 5.3-5.5 will be useful. It is worth to point out that the Atiyah-Hirzebruch spectral sequence collapses at most at the $E_5$-page in this context, so an (at least partial) computation of the differentials may not be completely out of sight.

We finish by pointing out that Artin groups of dihedral type are free-by-cyclic, and the Farrell-Jones conjecture has been recently verified for this class of groups \cite{BFW21}. Hence, all the computations in the left-hand side can be read in terms of algebraic $K$-theory of $RG$.


\textbf{Summary of contents}. In Section \ref{Sect:Prelim} we recall the main definitions about classifying spaces for families and Bredon homology, with special emphasis in the L\"{u}ck-Weiermann model and its associated Mayer-Vietoris sequence. In Section \ref{Sect:Artin} the main properties of Artin groups that will be used on the rest of the paper are studied. Then, in Section \ref{Sect:BredonArtin}, we carefully analyze the homomorphisms in the exact sequence and prove Theorem \ref{Thm:Bredon}; and in final Section \ref{Sect:concrete} we apply our results to describe different concrete examples.

\section{Preliminaries}
\label{Sect:Prelim}

In this section we state some notions of $G$-equivariant homotopy that will frequently appear in the rest of the paper.
The exposition will be sketchy and very focused to our goals; the reader interested in a thorough treatment of the subject is referred to \cite{TDieck} for the theory of $G$-$CW$-complexes and actions on them, to \cite{Luc05} for the theory of classifying spaces and to the first part of \cite{MiVa03} for Bredon homology.

\subsection{Classifying spaces for families}
\label{Sect:classify}
In this section we will briefly recall the notion of classifying space for a family of subgroups, which is the central object in the topological side of the Isomorphism Conjectures. Then we will review L\"{u}ck-Weiermann model and the definition of commensurator, which will be crucial in our computations.

\begin{defn}

Let $G$ be a discrete group, and $\mathcal{F}$ be a family of subgroups of $G$  closed under passing to subgroups and conjugation. A $G$-CW-complex $X$ is a \emph{classifying space for the family} $\mathcal{F}$ if for every $H\in \mathcal{F}$ the fixed-point set $X^H$ is contractible, and empty otherwise.

\end{defn}

The classifying space for the family $\mathcal{F}$ is usually denoted by $E_{\mathcal{F}}G$.
Moreover,  two models for $E_{\mathcal{F}}G$ are $G$-homotopy equivalent. A point is always a model for $E_{\mathcal{F}}G$ if $G\in \mathcal{F}$, and the closeness under subgroups implies that $E_{\mathcal{F}}G$ is always a contractible space.

If there is a family $\F$ of subgroups of $G$ with the previous closeness properties and a subgroup $H\leqslant G$, we denote by $\F\cap H$ the family whose elements are the intersections $F\cap H$, with $F\in \F$.
The family $\F\cap H$ of subgroups of $H$ is again closed under $H$-conjugation and taking subgroups.
In these conditions the action of $H$ over $E_{\F}G$ by restriction turns $E_{\F}G$ into a model for $E_{\F\cap H}H$.

The most important families of subgroups in this context are the trivial family $\mathcal F_{\{1\}}$, the family $\mathcal{F}_{Fin}$ of finite groups and the family $\mathcal{F}_{vc}$ of virtually cyclic groups of $G$; the classifying spaces for these families are respectively denoted by $EG$, $\underline{E} G$ and $\E G$. Observe that $\mathcal F_{\{1\}}\subseteq \mathcal{F}_{Fin}\subseteq \mathcal{F}_{vc}$, and that $\mathcal F_{\{1\}}=\mathcal{F}_{Fin}$ if and only if $G$ is torsion-free.
It is  also a standard argument to show that torsion-free virtually cyclic groups are cyclic (see for example Lemma 3.2 in \cite{Mac96}).
Then, for torsion-free groups, $\mathcal F_{vc}$ is the set of cyclic subgroups.

From now on we will describe the model of $\E G$ developed by L\"{u}ck-Weiermann in \cite{LW}, for the special families we are interested (the construction is indeed more general). Given a group $G$, and two subgroups $H$ and $K$,  we consider the equivalence relation generated by $H\sim K$ if $H\cap K$ has finite index in both $H$ and $K$.
Observe that if $H$ is in ${\mathcal F}_{vc}\setminus \mathcal{F}_{Fin}$ then $H\sim K$ if and only if $K$ is virtually cyclic and $H\cap K$ is infinite.
Also remark  that if $H\in \mathcal{F}_{Fin}$, then $H\sim K$ if and only if $K\in \mathcal{F}_{Fin}$.
The equivalence class of $H$ will be denoted by $[H]$.
This equivalence relation is preserved by conjugation, it can be defined $g^{-1}[H]g$ as $[g^{-1}Hg]$ for any $H\leqslant G$ and for any $g\in G$.

Now we can define the notion of \emph{commensurator}, central in this model and  in our paper:

\begin{defn}
\label{defn:comm}
Given an equivalence class $[H]$ of the relation $\sim$, the \emph{commensurator} of $[H]$ in $G$ is defined as the subgroup
$$ \textrm{Comm}_G[H]=\{g\in G\:|\ g^{-1}[H]g=[H]\}.$$
\end{defn}

In \cite{LW}, it is also defined a family of subgroups of $\textrm{Comm}_G[H]$ as:
$$\F [H]:=\{K<{ \textrm{Comm}_G[H]} \: |\ K\sim H \text{ or } |K|<\infty \}.$$
If $H$ is virtually cyclic, it is easy to check that this family is closed under taking subgroups and conjugation in {$\textrm{Comm}_G[H]$}. {We remark that the commensurator of $[H]$ is sometimes denoted by $N_G[H]$ in the literature}.

Now we have all the ingredients needed for building L\"{u}ck-Weiermann model:

\begin{thm}
\label{maintheorem}{\rm (\cite{LW}, Theorem 2.3)} We denote by $I$ a complete set of representatives of the $G$-orbits (under conjugation) of equivalence classes $[H]$ of infinite virtually cyclic subgroups of $G$, and we choose, for every $[H]\in I$, models for the classifying spaces $\underline{E}  \Comm_G[H]$ and $E_{\F [H]}\Comm_G[H]$.
We also choose a model for $\underline{E}  G$. Consider the $G$-pushout:
$$
\xymatrix{  \coprod_{[H]\in I}G\times_{\Comm_G[H]}\underline{E}  \Comm_G[H] \ar[r]^{\hspace{2cm} i} \ar[d]^{\coprod_{[H]\in I}id_G\times_{\Comm_G[H]}f_{[H]}}  & \underline{E} G \ar[d] \\
\coprod_{[H]\in I}G\times_{\Comm_G[H]}E_{\F [H]}\Comm_G[H]  \ar[r] & X }
$$
where  $f_{[H]}$ is a cellular $\Comm_G[H]$-map for every $[H]\in I$ and $i$ is the inclusion.
 In these conditions, $X$ is a model for $\E G$.
\end{thm}

In practice, this theorem implies that the existence of good models for the proper classifying space of the commensurators and $G$, and also of the classifying spaces with respect to the families $\F [H]$ will lead to the knowledge of good models for $\E G$. Moreover, the push-out implies the existence of a long exact sequence in Bredon homology, and dimensional consequences that we will analyze in next section.

\subsection{Bredon homology}
\label{Sect:Bredon}

{In this subsection we will briefly review the main definitions concerning Bredon homology. We follow the topological concise approach from \cite{San08}, which we use in our computations}.

{Consider a discrete group $G$, $\mathcal{F}$ a family of groups which is closed under conjugation and taking subgroups.  Let $O_{\mathcal{F}}(G)$ be the \emph{orbit category}  whose objects are the homogeneous spaces $G/K$, $K\subset G$ with $K\in\mathcal{F}$, and whose morphisms are the $G$-equivariant maps.
Then a \emph{left Bredon module} $N$ over $O_{\mathcal{F}}(G)$ is a covariant functor $$N:O_{\mathcal{F}}(G)\ra \textbf{Ab},$$ where $\textbf{Ab}$ is the category of abelian groups}.

{Let $N$ be  a left Bredon module and $X$ a $G$-CW-complex, and assume that all the stabilizers of the $G$-action belong to the family $\mathcal{F}$. Then the \emph{Bredon chain complex} $(C_n^{\mathcal{F}}(X,N),\Phi_n)$ can be defined in the following way.
For every $d\geq 0$, consider a set $\{e_i^d\}_{i\in I}$ of representatives of orbits of $d$-cells in $X$, and denote by $\stab(e_i^d)$ the stabilizer of $e_i^d$.
Then we define the \emph{n-th group of Bredon chains} as $C_n^{\mathcal{F}}(X,N)=\bigoplus_{i\in I} N(G/\stab(e_i^d))$}.

{Consider now a $(d-1)$-face of $e_i^d$, which can be given as $ge$ for a certain $(d-1)$-cell $e$.
Then we have an inclusion of stabilizers $g^{-1}\stab(e_i^d)g\subseteq \stab(e)$.
As $g^{-1}\stab(e_i^d)g$ and $stab(e_i^d)$ are isomorphic, the previous inclusion induces an equivariant $G$-map $f\colon G/\stab(e_i^d)\ra G/\stab(e)$.
In turn, as $N$ is a functor, we have an induced homomorphism $N(f)\colon N(G/\stab(e_i^d))\ra N(G/\stab(e))$.
Taking into account that the boundary of $e_i^d$ can be written as $\partial e^d_i=\sum_{j=1}^n g_j e_j^{d-1}$ for certain $g_j\in G$ and using linear extension to all representatives of equivariant $d$-cells, we obtain a differential $\Phi_d\colon C_d^{\mathcal{F}}(X,N)\ra C_{d-1}^{\mathcal{F}}(X,N)$ for every $d>0$.
So we have the following definition:}

\begin{defn}
{The homology groups of the chain complex $(C_i^{\mathcal{F}}(X,N),\Phi_i)$ will be denoted by $H_i^{\mathcal{F}}(X,N)$ and called \emph{Bredon homology groups} of $X$ with coefficients in $N$ with respect to the family $\mathcal{F}$}.

We define $H_i^{\mathcal{F}}(G,N)$, the {\emph Bredon homology groups} of $G$ with coefficients in $N$  with respect to the family $\mathcal{F}$ as $H_i^{\F}(E_{\mathcal{F}}G,M)$.
\end{defn}

{These groups are preserved under $G$-equivariant homotopy equivalence}.

\textbf{Notation}. When $\mathcal{F}$ is the family of finite groups, we use indistinctly the notations $\underline{E}G$ or $E_{\mathcal{F}}G$, and similarly when $\mathcal{F}$ is the family of virtually cyclic groups and $\E G$ and $E_{\mathcal{F}}G$ notations for the corresponding classifying space. If $\mathcal{F}$ is the family that only contains the trivial group, the superindex in the Bredon homology will be supressed, as it is ordinary homology in this case.
It is worth noticing that there is an algebraic definition of $H_*^{\mathcal{F}}(G,M)$, however there is an isomorphism $H_*^{\mathcal{F}}(G,M)\simeq H_*^{\mathcal{F}}(E_{\mathcal{F}}G,M)$ between the algebraic and the topological definitions of Bredon homology \cite[page 15]{MiVa03}.
To not overload the paper with unnecessary notation, we commonly denote this homology groups by $H_*^{\mathcal{F}}(G,M)$, although as said above, we mainly deal with the topological definition.

The cyclic group of order $n$ will be denoted by $C_n$. When we want to consider its ring structure we might use  $\Z /{\bf n}$. When $n$ is prime, we might use $\mathbb{F}_n$ to emphasize its field structure.

\section{Artin groups of dihedral type}
\label{Sect:Artin}
In this section we present the main features of the Artin groups of dihedral type that we will need in the remaining of the paper. We start with the definition of the groups:

\begin{defn}
\label{Defn:Artindih}
Let $n\geq 1$. By $\mathrm{prod}(x,y;n)$ we denote the word of length $n$ that alternates $x$ and $y$ and starts with $x$.
For example, $\mathrm{prod}(x,y;3)=xyx$ and $\mathrm{prod}(x,y;4)=xyxy$. With this notation, a {\it dihedral Artin group of type $n$} is the group $$A_n=\gp{a,b}{\mathrm{prod}(a,b;n)=\mathrm{prod}(b,a;n)}.$$
\end{defn}
The name ``dihedral" comes from the associated Coxeter group, $$\gp{a,b}{a^2=b^2=1, \mathrm{prod}(a,b;n)=\mathrm{prod}(b,a;n)}$$ which is the dihedral group of order $2n$. Dihedral Artin groups are torsion-free, even more $A_n\cong F_{n-1}\rtimes \Z$, where $F_k$ is a free group of rank $k$. To see this, one can check that the kernel of $A_n\to \Z$, $a,b\mapsto 1$ is free on rank $n-1$. In particular, Dihedral Artin groups satisfy the Farrell-Jones conjecture \cite{BFW21}.

We are interested on understanding the commensurators of the virtually cyclic subgroups of $A_n$, for that we will use that $A_n$ is also a central extension of a virtually free group.

As indicated above an important ingredient of our calculations will be the description of some commensurators inside these Artin groups of subgroups from the family of virtually cyclic groups (which in this case turn to be just cyclic groups). We start by observing that for any  virtually cyclic subgroup $H$ of a group $G$, and any $h\in H$ of` infinite order, we have that
$$\Comm_G[H]=\Comm_G(\gen{h})=\{g\in G\mid gh^mg^{-1}=h^n \text{ for some }n,m\in \Z-\{0\}\}.$$

We will be interested in computing commensurators up to isomorphism, and we will use the fact that commensurators of conjugated subgroups are conjugated.
Let us denote by $\pi$ the natural
projection map $\pi\colon A_{n}\to \ol{A_{n}}\coloneq A_{n}/Z(A_{n})$. Then, for every $g\in A_{n}$,
$$\pi(\Comm_{A_{n}}[\gen{g}])\leqslant \Comm_{\ol{A_{n}}}[\gen{\pi(g)}],$$
and
$$Z(A_{n})\leqslant \Comm_{A_{n}}[\gen{g}].$$

We will prove the following.
\begin{lem}\label{lem:commensurators}
Let $A_n$ be a dihedral Artin group and $g\in A_n$ of infinite order. Then
\begin{enumerate}
\item If $\langle g\rangle \cap  Z(A_n)\neq \{1\}$ then $\Comm_{A_n}[\langle g\rangle]= A_n$ and $Z(A_n)\in [H]$,
\item If $\langle g\rangle \cap Z(A_n)= \{1\}$ then $\Comm_{A_n}[\langle g\rangle]\cong \mathbb{Z}^2$ and there is $\langle g'\rangle \in [\langle g\rangle]$  that is a direct factor of $\Comm_{A_n}[\langle g\rangle]$.
\end{enumerate}
\end{lem}

\begin{proof}

It is a well-known fact that the centers of the dihedral Artin groups have different shape depending on the parity of $n$. Hence, we should divide the study of the commensurators of their cyclic subgroups in two subcases, even and odd.

{\bf Case $n$ even: }
In the group $$A_{2n}=\gp{a,b}{\mathrm{prod}(a,b;2n)=\mathrm{prod}(b,a;2n)},$$
it is known that $Z(A_{2n})=\gen{(ab)^{n}}$.
Moreover, $A_{2n}\cong \gp{x,y}{x^{-1}y^nx=y^n}$ via the isomorphism defined by $x\to b$, $y\mapsto ba$.
Therefore $A_{2n}/Z(A_{2n})=\gp{x,y}{x^{-1}y^nx=y^n,y^n =1}=\gp{\ol{x},\ol{y}}{\ol{y}^n=1}\cong C_{\infty}*C_{n}$.

Since $\ol{A_{2n}}$ is hyperbolic (even more, virtually free), the commensurator of an infinite order element  is virtually cyclic. (See for example \cite[Theorem 2]{Arz} bearing in mind that infinite cyclic subgroups of hyperbolic groups are quasi-convex).

Suppose first that $\pi(g)$ has infinite order. Then $\Comm_{\overline{A_{2n}}}[\langle \pi(g)\rangle]$ is infinite and virtually cyclic and hence $\pi (\Comm_{A_{2n}}[\gen{g}])$ is infinite and virtually cyclic.
Since all infinite virtually  cyclic subgroups of $\ol{A_{2n}}$ are infinite cyclic  we have that  $\pi (\Comm_{A_{2n}}[\gen{g}])$ is infinite cyclic. Therefore,
$\Comm_{A_{2n}}[\gen{g}]$ is a central extension of $\Z$ by $\Z$ and hence $\Comm_{A_n}[\gen{g}]\cong \Z^2$.
We can take $g'$ as a pre-image of a generator of $\pi (\Comm_{A_{2n}}[\gen{g}])$.

Suppose now that $\pi(g)$ has finite order.
Thus $\langle g \rangle \cap Z(A_{2n})$ is non-trivial and infinite,  and thus $[\langle g \rangle]= [Z(A_{2n})]$ and $\Comm_{A_{2n}}[\gen{g}]=\Comm_{A_{2n}}Z(A_n)=A_{2n}.$

{\bf Case $n$ odd:}
In the group $$A_{2n+1}=\gp{a,b}{\mathrm{prod}(a,b;2n+1)=\mathrm{prod}(b,a;2n+1)}$$ we have that $Z(A_{2n+1})=\gen{(ab)^{2n+1}}$.
Moreover, $A_{2n+1}\cong \gp{x,y}{xy^n=y^{n+1}x^{-1}}$ via the isomorphism defined by $x\to b$, $y\mapsto ab$.
Therefore
\begin{align*}
A_{2n+1}/Z(A_{2n+1})&=\gp{x,y}{xy^nx=y^{n+1},y^{2n+1}=1}=\gp{x,y}{(xy^n)^2=1=y^{2n+1}}\\
&=\gp{\ol{z},\ol{y}}{\ol{z}^2=1=\ol{y}^{2n+1}}\cong C_{2}*C_{2n+1},
\end{align*}
where $\ol{z}$ denotes the class of $xy^nZ(A_{2n+1})$.

Since $\ol{A_{2n+1}}$ is hyperbolic and infinite virtually cyclic subgroups of $\ol{A_{2n+1}}$ are infinite cyclic we get, arguing as above, that if $\pi(g)$ has infinite order then $\Comm_{A_{2n+1}}[\gen{g}]\cong \Z^2$.

Suppose now that $\pi(g)$  has finite order. Arguing as above, $\Comm_{A_{2n+1}}[\gen{g}] = \Comm_{A_{2n+1}}[Z(A_{2n+1})] = A_{2n+1}$.
\end{proof}


We  recall  the ordinary homology of these groups, which will be important in the remaining sections of the paper.

\begin{prop}
\label{OrdHom}
For $n\geq 2$.

If $n$ is even, we have  $H_0(A_n)=\mathbb{Z}$, $H_1(A_{n})=\mathbb{Z}\oplus \mathbb{Z}$, $H_2(A_{n})=\mathbb{Z}$ and $H_i(A_{n})=0$ for $i>2$.

 If $n$ is odd, we have $H_0(A_n)=\mathbb{Z}$, $H_1(A_{n})=\mathbb{Z}$ and $H_i(A_{n})=0$ for $i>1$.

\end{prop}

\begin{proof}
As dihedral Artin groups are one-relator groups, the Cayley complex (associated to the one-relator presentation) gives a 2-dimensional model for $K(A_n,1)$ (see \cite{Lyndon} or \cite{CMW04} for more general Artin groups).
 As $A_n$ is not free, this implies that $\textrm{gd }A_n=\textrm{cd }A_n=2$ for every $n$, and then in particular $H_i(A_n)=0$ for $i\geq 3$. Moreover, as $K(A_n,1)$ is connected, $H_0(A_n)=\mathbb{Z}$ for every $n$.
The formulae for $H_1$ follow from performing abelianization to the groups. Finally, the results for the Schur multiplier $H_2$ are a consequence of Theorem 3.1 in \cite{CE09}.
\end{proof}




\section{Bredon homology of Artin groups of dihedral type}
\label{Sect:BredonArtin}

In this section we start with the computation of homological invariants of the Artin groups of dihedral type, which is the main goal of this paper. The computation of the Farrell-Jones homology in the case of a general ring is usually very complicated. A powerful tool to undertake this computation is the $G$-equivariant version of the Atiyah-Hirzebruch spectral sequence, which is a spectral sequence of the 1st and 4th quadrant. In our case of interest, the $E_2$-page of this sequence is given by Bredon homology of the classifying space $\E G$, and the $E_{\infty}$-page encodes the Farrell-Jones homology. Let us make the last statement more precise.

Let $G$ be a discrete group, $\mathcal{F}$ its family of virtually cyclic subgroups, $R$ a ring. For every $q\in\mathbb{Z}$, denote by $K_q(R[-])$ the covariant module over the orbit category $O_{\mathcal F}(G)$ that sends every (left) coset  $G/H$ to the $K$-theory group $K_q(R[H])$. In these conditions, the $E_2$-page of the Atiyah-Hirzebruch  spectral sequence that we will use is defined as $E_2^{p,q}=H_p^{\mathcal{F}}(\E G,K_q(R[-]))$, for every $p\geq 0$ and $q\in\mathbb{Z}$, and converges to $H^G_{p+q}(\E G,\mathbf{K}(R))$ the $(p+q)$-th group of Farrell-Jones homology of $\E G$ with coefficients in $R$. An excellent source for more information about this sequence is \cite{LR05}.

Although the path {to} the computation is very clear, in general it is difficult to obtain explicit formulae for the Bredon
homology of these classifying spaces. Different reasons for this are the complexity of the models for $\E G$ and the fact that the exact values of $K_q(R[H])$ are only known for very special instances of $R$ and $H$. In fact, taking $R=\mathbb{Z}$ and $H$ the trivial group, the groups $K_q(\mathbb{Z})$ are not completely listed, as their value in some cases depend on the solution of the Vandiver conjecture, which remains unsolved. See \cite{Wei05}.

In this section we show that for Artin groups of dihedral type it is possible to describe to some extent the aforementioned Bredon homology groups, with coefficients in $K_q(R[-])$ for a general ring $R$ (see Theorem \ref{Thm:Bredon}); of course, the result strongly depends on the concrete shape of the $K$-theory groups of the group rings involved. The key to our computations is the description of the commensurators of the virtually cyclic subgroups of the groups $A_n$ (see previous section), and mainly the Mayer-Vietoris sequence associated to the push-out of L\"{u}ck-Weiermann model (Theorem \ref{maintheorem}), that was explicitly stated by Degrijse-Petrosyan in Section 7 of \cite{DP}, for the cohomological case. We offer here the homological version, which is the one needed in our context. To avoid confusions, we denote Bredon homology with respect to the family of virtually cyclic groups as $H_*^{vc}$ from now on.

\begin{prop}{\rm (\cite[Proposition 7.1]{DP}, see also \cite[Theorem 2.3]{LW})}
\label{mainMV} Let  $M$ be a left Bredon module over $O_{\F_{vc}}(A_n)$.
There is an exact sequence:

$$\ldots \stackrel{}{\rightarrow} H^{vc}_{i+1}(A_n,M) \stackrel{g_1^{i+1}}{\rightarrow} \bigoplus_{[H]\in I}H_i^{Fin\cap \Comm_{A_n}[H]}(\Comm_{A_n}[H],M) \stackrel{g_2^i}{\rightarrow} $$

$$\stackrel{g_2^i}{\rightarrow} (\bigoplus_{[H]\in I}H_i^{\mathcal{F}[H]}(\Comm_{A_n}[H],M))\oplus H_i^{Fin}(A_n,M) \stackrel{g_3^i}{\rightarrow} H^{vc}_i(A_n,M) \rightarrow \ldots$$

\end{prop}

Before we undertake our computations of the homology groups, we will compute their geometric dimension $\gde$ with respect to the family of virtually cyclic groups.

\begin{prop}
\label{gd}

The geometric dimension of $A_n$ with respect of the family of virtually cyclic groups is 3.

\end{prop}

\begin{proof}
As seen in Lemma \ref{lem:commensurators}, $A_n=\langle a,b\,|\,\textrm{prod}(a,b;n)=\textrm{prod}(b,a;n) \rangle$, contains subgroups isomorphic to $\mathbb{Z}\oplus\mathbb{Z}$ (for example $\Comm_{A_n}(\langle a\rangle))$. According to Example 5.21 in \cite{LW}, this implies that $\gde A_n\geq 3.$

On the other hand, as the Artin groups $A_n$ are one-relator, Corollary 3 in \cite{Deg16} implies that $\gde A_n\leq 3$, and then $\gde A_n=3$.
\end{proof}

As stated, the strategy to compute $H_i^{vc}(A_n,M)$ goes through describing the different elements of the exact sequence of Proposition \ref{mainMV}.
We first understand the terms $H_i^{\mathcal{F}[H]}$ on Subsection \ref{HiF[H]}.
Then, in Subsection \ref{Ktheory}, we give a more concrete description of the different homology groups when we take coefficients in the $K$-theory.
Finally, in Subsection \ref{morphism} {we study the homomorphisms} $g_1$ and $g_2$ of the exact sequence of Proposition \ref{mainMV}  and prove our main theorems.

\subsection{Computing {the homology of the commensurators}}\label{HiF[H]}
Let us start the main calculations of this section.
It is clear from the previous Mayer-Vietoris sequence that the computations of the ordinary homology of the commensurators
(which includes the homology of $A_n$, as the center of $A_n$ is virtually cyclic)
 and the homology of the commensurators with coefficients in $\mathcal{F}[H]$ will give us valuable information about the Bredon homology of $A_n$ with respect to the family of virtually cyclic subgroups,
so we will perform these calculations in the sequel.
We start with the ordinary homology, which is straightforward and depends on the shape of the commensurators:

\begin{itemize}

\item {If} $\textrm{Comm}_{A_n}[H]\simeq\mathbb{Z}\oplus \mathbb{Z}$, we have $H_0(\textrm{Comm}_{A_n}[H])=H_2(\textrm{Comm}_{A_n}[H])=\mathbb{Z}$, $H_1(\textrm{Comm}_{A_n}[H])=\mathbb{Z}\oplus \mathbb{Z}$, $H_i(\textrm{Comm}_{A_n}[H])=0$ for $i>2$.

\item {If} $\textrm{Comm}_{A_n}[H]\simeq A_n$, see Proposition \ref{OrdHom}.

\end{itemize}

We concentrate now in the case of $\mathcal{F}[H]$, which we recall is the family of subgroups of $\Comm_{A_n}[H]$ that are either finite or commensurable with $H$. Concretely, we intend to compute $H_i^{\mathcal{F}[H]}(\textrm{Comm}_{A_n}[H],M)$ for every non-trivial cyclic subgroup $H$ of $A_n$.
We have two cases, either $H\cap Z(A_n)$ is trivial or not. We will use the following Convention throughout the paper without explictly refering to it.

\begin{conv}
For $[H]$ a class of infinite cyclic subgroups, we can always take a representative $H$ that is normal in $\Comm_{A_n}[H]$.

If $H\cap Z(A_n)$ is trivial, then Lemma \ref{lem:commensurators} states  that $H\unlhd \Comm_{A_n}[H]$. Moreover, by  Lemma \ref{lem:commensurators} we assume that  $H$ is a direct factor of  $\Comm_{A_n}[H]$.

If $H\cap Z(A_n)$ is non-trivial, then $[H]=[Z(A_n)]$ and we will assume in this case that we choose $Z(A_n)$ as the representative of $[H]$, and thus  $H\unlhd \Comm_{A_n}[H]$ again.

With this convention,  the projection $\pi\colon\textrm{Comm}_{A_n}[H]\rightarrow \textrm{Comm}_{A_n}[H]/H$ is well-defined and $\Comm_{A_n}[H]/H$ is isomorphic to $C_\infty$, $C_\infty*C_n$ or $C_2*C_{2n+1}$.
\end{conv}

Let now $M$ be a module over the orbit category of $\textrm{Comm}_{A_n}[H]$ with respect to $\mathcal{F}[H]$, and $\pi^{-1}M$ the induced module over the orbit category of $\Comm_{A_n}[H]/H$ with respect to the family of finite groups.
That is, for $K\leqslant \textrm{Comm}_{A_n}[H]/H$ finite {it is defined}
$$\pi^{-1}M((\Comm_{An}[H]/H)/K)\coloneqq M(\Comm_{A_n}[H]/\pi^{-1}(K)).$$ {Observe that this assignation gives rise to a natural transformation of functors $\pi^{-1}M\rightarrow M$. In the next result, which is essentially \cite[Lemma 4.2]{DP}, but stated for homology instead of cohomology, we will see that this natural transformation induces an isomorphism in Bredon homology}. It will be a powerful tool in our computations.
\begin{prop}
\label{induced}
Let $H$ be an infinite cyclic subgroup of $A_n$ normal in its commensurator.
For every $n\geq 0$, and every module $M$ over the orbit category of $\Comm_{A_n}[H]$ there is an isomorphism
{$$H_i^{Fin}(\Comm_{A_n}[H]/H,\pi^{-1} M)\simeq H_i^{\mathcal{F}[H]}(\Comm_{A_n}[H],M).$$}
Moreover, every model for $\uE(\Comm_{A_n}[H]/H)$ is a model for $E_{\mathcal{F}[H]}\Comm_{A_n}H$, with the action induced by the quotient map $\Comm_{A_n}[H]\to \Comm_{A_n}[H]/H$.
\end{prop}
\begin{proof}
The argument here is taken from the proof of \cite[Lemma 4.2]{DP}. We write it here to make our paper more self-contained.

The projection $\pi\colon \Comm_{A_n}[H]\to \Comm_{A_n}[H]/H$ maps the family $\mathcal{F}[H]$ onto the family $Fin$ of finite subgroups of the quotient $\textrm{Comm}_{A_n}[H]/H$.
Moreover, the pre-image $\pi^{-1}(K)$ for any finite group $K$ of $\Comm_{A_n}[H]/H$ lies in $\mathcal{F}[H]$.
Therefore, $\uE \Comm_{A_n}[H]/H$ is a model for $E_{\mathcal{F}[H]}\textrm{Comm}_{A_n}[H]$, with the action induced by the projection $\pi$.

Consider now the spectral sequence associated to the short exact sequence
$1\to H\to \Comm_{A_n}[H]\stackrel{\pi}{\to} \Comm_{A_n}[H]/H\to 1$ for homology \cite{Mar02}.
For every module $M$ over the orbit category of $\Comm_{A_n}[H]$ we have
$$E^{p,q}_2 (M)= H_p^{{Fin}}(\Comm_{A_n}[H]/H, H_q^{\mathcal{F}[H]\cap \pi^{-1}(-)}(\pi^{-1}(-), M)),  $$
which converges to $E^{p,q}_{\infty} (M)=H_{p+q}^{\mathcal{F}[H]}(\Comm_{A_n}[H], M)$.

Observe that $E^{p,q}_2$ is trivial for $q\geq 1$, as for every finite subgroup $K<\Comm_{A_n}[H]/H$, $\pi^{-1}(K)$ belongs to the family  $\mathcal{F}[H]\cap \pi^{-1}(K)$, and then $H_q^{\mathcal{F}[H]\cap \pi^{-1}(-)}(\pi^{-1}(-),M)$ is zero. Thus, as $E^{i,0}_2 (M)=H_i^{Fin}(\Comm_{A_n}[H]/H,
\pi^{-1}M)$
in the 0-th row of the sequence, we have  {$$H_i^{Fin}(\Comm_{A_n}[H]/H,\pi^{-1} M)\simeq H_i^{\mathcal{F}[H]}(\Comm_{A_n}[H],M).$$} for every $i\geq 0.$
We are done.
\end{proof}

 The family of finite subgroups is easier to deal with than the family $\mathcal{F}[H]$,  taking account of the previous proposition, the next step is to construct the corresponding classifying spaces for proper actions of the commensurators in $A_n$ modulo the corresponding subgroups. We have the following:

\begin{itemize}

\item If $\textrm{Comm}_{A_n}[H]\simeq\mathbb{Z}\oplus \mathbb{Z}$, then there is a representative of the class $[H]$ that can be identified with one of these copies of $\Z$; hence, we may assume that the inclusion $H\hookrightarrow \textrm{Comm}_{A_n}[H]=\mathbb{Z}\oplus \mathbb{Z}$ is the inclusion of the first factor, and so $\textrm{Comm}_{A_n}[H]/H$ is isomorphic to $\mathbb{Z}$. Then a model for $\underline{{E}}(\textrm{Comm}_{A_n}[H]/H)$ is the straight line, and the action is by shifting.

\item If $\textrm{Comm}_{A_n}[H]\simeq A_n$, then $Z(A_n)$ is a representative of the class $[H]$.
We have seen that $\textrm{Comm}_{A_n}[H]/H=A_n/Z(A_n)$ is an amalgamated product of two cyclic groups, depending its concrete shape on the parity of $n$.
In this case a tree model for ${\uE}(\textrm{Comm}_{A_n}[H]/H)$ can be explicitly constructed.

If $n$ is even, we have that $A_{n}/Z(A_{n})\cong C_{\infty}*C_{n}$.
Denote $C_\infty*C_{n}$ by $\overline{A_n}$.
Let $s$ be a generator of $C_\infty$.
Then Bass-Serre theorem guarantees that the graph with vertex set $\overline{A_n}/C_n$, edge set $\overline{A_n}$ and incidence maps $\iota(g) = gC_\infty$ and $\tau(g)=gsC_\infty$ is a $\overline{A_n}$-equivariant oriented tree.

If $n$ is odd, we have that $A_{n}/Z(A_{n})\cong C_{2}*C_{n}$. Denote $C_2*C_{n}$ by $\overline{A_n}$.
Then Bass-Serre theorem guarantees that the graph with vertex set $\overline{A_n}/C_2 \sqcup \overline{A_n}/C_n$, edge set $\overline{A_n}$ and incidence maps $\iota(g) = gC_2$ and $\tau(g)=gC_n$ is a $\overline{A_n}$-equivariant oriented tree.

Note that in both cases, the isotropy groups are the finite subgroups of $\overline{A_n}$ and they fix exactly a vertex; and therefore these are  models for $\uE \overline{A_n}$.
\end{itemize}

As the classifying spaces of the commensurators that we have described are all 1-dimensional, we will use the following result of Mislin:

\begin{lem}[\cite{MiVa03}, Lemma 3.14]
\label{1dim}
Suppose that for a family $\mathcal{F}$ of subgroups of a group $G$ there is a tree model $T$ for $E_{\mathcal{F}}G$. Let $S_e$ be the stabilizer of the edge $e\in T$ and $S_v$ the stabilizer of a vertex. Let $N$ be a coefficient module. Then $H_i^{\mathcal{F}}(G,N)=0$ for $i>1$ and there is an exact sequence:

$$ 0\rightarrow H_1^{\mathcal{F}}(G,N)\rightarrow \bigoplus_{[e]} N(G/S_e) \rightarrow \bigoplus_{[v]} N(G/S_v)\rightarrow H_0^{\mathcal{F}}(G,N)\rightarrow 0,$$ where $[e]$ and $[v]$ run over the $G$-orbits of edges and vertices of $T$, respectively.
\end{lem}

\begin{rem}
\label{diff}
It is interesting to remark that the middle map $\bigoplus_{[e]} N(G/S_e) \rightarrow \bigoplus_{[v]} N(G/S_v)$ is induced by the (formal) border operation defined by $\partial[e]=[v_1]-[v_2]$, being $v_1$ and $v_2$ vertices of a representative $e$ of $[e]$. The tree is assumed to be oriented.
\end{rem}

In particular we obtain:

\begin{cor}
\label{zerohomology}
For every $n\geq 2$, $i\geq 2$, $H\leqslant  A_n$ infinite cyclic and every module $M$ over the orbit category with respect to the family $\mathcal{F}[H]$, we have $H_i^{\mathcal{F}[H]}(\Comm_{A_n}[H],M)=0$.
\end{cor}
\begin{proof}
Recall that by Proposition \ref{induced} $H_i^{\mathcal{F}[H]}(\Comm_{A_n}[H],M)\simeq H_i^{Fin}(\Comm_{A_n}[H]/H,\pi^{-1}M)$. As $\uE\Comm_{A_n}[H]/H$ is 1-dimensional, the result follows by Lemma \ref{1dim}.
\end{proof}

So it remains to compute $H_0^{\mathcal{F}[H]}(\textrm{Comm}_{A_n}[H],M)$ and $H_1^{\mathcal{F}[H]}(\textrm{Comm}_{A_n}[H],M)$. We assume here a general coefficient module $M$, but the reader may keep in mind that our case of interest is $M(A_n/-
)=K_q (R[-])$, $q\in\Z$.

\begin{prop}\label{homology0and1}
Let $A_n =\langle a, b \,\textrm{ }|\textrm{ } \mathrm{prod}(a,b;n) = \mathrm{prod}(b,a;n)\rangle$
 be a dihedral Artin group, let $H\leqslant A_n$ be an infinite cyclic group that is normal in its commensurator, let $\pi\colon \Comm_{A_n} [H]\to \Comm_{A_n}[H]/H$ be the natural projection.
Let $M$ be a module over the orbit category of $\Comm_{A_n}[H]$.
Then
\begin{enumerate}
\item[(i)] If $H\cap Z(A_n)$ is trivial, then
$$H_0^{\mathcal{F}[H]}(\Comm_{A_n}[H],M)=H_1^{\mathcal{F}[H]}(\Comm_{A_n}[H],M)=M(\Comm_{A_n}[H]/H).$$
\item[(ii)] If $H=Z(A_{n})$, $n=2k+1$  and if $f\colon M(A_{n}/Z(A_n))\to M(A_n/\langle b(ab)^k\rangle ) \oplus M(A_n/\langle ab \rangle)$ {is} induced by the  natural projections,  we have
$$H_0^{\mathcal{F}[H]}(A_{n},M)= \mathrm{coker} f\quad \text { and } \quad H_1^{\mathcal{F}[H]}(A_{n},M)= \ker  f.$$

\item[(iii)] If $H=Z(A_n)$ and $n=2k$, we have
 $$H_0 ^{\mathcal{F}[H]}(A_{n},M)=M(A_{n}/\langle ab \rangle)\quad \text{  and  }\quad H_1^{\mathcal{F}[H]}(A_{n},M)= M(A_{n}/Z(A_n)).$$
\end{enumerate}
Note that in cases (ii) and (iii), $A_n=\Comm_{A_n}[H]$.
\end{prop}

\begin{proof}
In the case $H\cap Z(A_n)$ is trivial, the classifying space $\uE\Comm_{A_n}[H]/H$ is the real line,  there is one orbit of edges and one orbit of vertices, and the action is free. Then by Remark \ref{diff}, the central map in that exact sequence is trivial and by Lemma \ref{1dim},
$$H_1^{Fin}(\Comm_{A_n}[H]/H,\pi^{-1}M)= \pi^{-1}M((\Comm_{A_n}[H]/H)/\{1\})$$ and
$$H_0^{Fin}(\Comm_{A_n}[H]/H,\pi^{-1}M)=\pi^{-1}M((\Comm_{A_n}[H]/H)/\{1\}).$$
Using  Proposition \ref{induced}, case (i) follows.

Let us discuss the case $H = Z(A_n)$. Now $\Comm_{A_n}[H]/H=A_n/Z(A_n)$ is an amalgamated product of two finite cyclic groups or of a finite cyclic and an infinite cyclic group, depending on the parity of $n$. In both cases, there is a tree model for the classifying space for proper actions of the commutator modulo $H$. Denote $\Comm_{A_n}[H]/H$ by $\overline{A_n}$ for brevity.

We consider first the case $n=2k+1$.
Here $\overline{A_{n}}=S*L$, with $S=\langle b(ab)^k Z(A_{n})\rangle \cong C_2$ and $L=\langle ab Z(A_{n}) \rangle \cong C_{n}$.
The Bass-Serre tree has two orbits of vertices, with stabilizers conjugated to $S$ or $L$, and one free orbit of edges.
Then, by Lemma \ref{1dim} 

\begin{equation}
\label{MVodd}
H_1^{Fin}(\overline{A_n},\pi^{-1} M)\hookrightarrow \pi^{-1}M(\overline{A_{n}}/\{1\})\stackrel{f}{\rightarrow}
 \end{equation}
$$\stackrel{f}{\rightarrow} \pi^{-1}M(\overline{A_{n}}/S)\oplus \pi^{-1}M(\overline{A_{n}}/L)\twoheadrightarrow H_0^{Fin}(\overline{A_{n}},\pi^{-1}M).
$$
 Hence, we obtain that $$H_0^{Fin}(\overline{A_{n}},\pi^{-1}M)=(\pi^{-1}M(\overline{A_{n}}/S)\oplus \pi^{-1}M(\overline{A_{n}}/L))/\textrm{Im }f$$ and $$H_1^{Fin}(\overline{A_{n}},\pi^{-1}M)=\textrm{Ker }f.$$
  Observe that the two components of $f$ are induced by the images of the projections $\overline{A_{n}}/\{1\}\rightarrow \overline{A_{n}}/S$ and $\overline{A_{n}}/\{1\}\rightarrow \overline{A_{n}}/L$ by the functor $\pi^{-1}M$. The case (ii) follows by applying Proposition \ref{induced}.

Consider now the  case  $n=2k$.
In this situation $\overline{A_{n}}=A_{n}/Z(A_{n})$ is a free product $G=S\ast L$, with $S=\langle b Z(A_n) \rangle \simeq C_{\infty}$ and $L=\langle ab Z(A_n) \rangle\simeq C_n$.
We denote by $s$  the element  $b Z(A_n)$.
The Bass-Serre tree for this group has
one orbit of vertices with stabilizers conjugate to $L$ and one orbit of edges with trivial stabilizers.

The exact sequence of Proposition \ref{1dim} has the following form:
\begin{equation}
\label{MVeven}
H_1^{Fin}(\overline{A_n},\pi^{-1}M)\hookrightarrow M(\overline{A_{n}}/\{1\}) \stackrel{f}{\rightarrow}  \pi^{-1}M(\overline{A_{n}}/L) \twoheadrightarrow H_0^{Fin}(\overline{A_{n}},\pi^{-1}M).
\end{equation}
The function $f$ is induced by  $g\{1\}\in A_{2n}/\{1\} \mapsto gsL- gL$ by the functor $M$.
And hence $f$ is trivial.

Hence, we deduce that $$H_1^{Fin}(\overline{A_{n}},\pi^{-1}M)= \pi^{-1}M(\overline{A_{n}}/\{1\})$$ and $$H_0^{Fin}(\overline{A_{n}},\pi^{-1}M)=\pi^{-1}M(\overline{A_{n}}/L).$$
The result now follows by applying Proposition \ref{induced}.
\end{proof}

Observe that taking into account that $\overline{A_{2n}}$ is one-relator, this computation agrees with the result of Corollary 3.23 in \cite{MiVa03}, in the case of $M=R_\mathbb{C}$, the complex representation ring.

\subsection{ Coefficients in  $K$-theory groups}
\label{Ktheory}
From the point of view of Farrell-Jones conjecture, the case of interest is the coefficients in the $K$-theory $K_q(R)$ of the group ring. We now describe all the homology groups of the exact sequence of Proposition \ref{mainMV} with  coefficients in this module, except the groups $H^{vc}(A_n,K_q(R[-]))$ which will be studied in the next section.

We begin recalling the following definiton.
\begin{defn}\label{defn:regular}
A {\it regular} ring is a
 commutative noetherian ring such that in the localization at every prime ideal, the Krull dimension of the maximal ideal is equal to the cardinal of a minimal set of generators.
\end{defn}
Examples of regular rings include fields (of dimension zero) and Dedekind domains. If $R$ is regular then so is the polynomial ring $R[x]$, with dimension one greater than that of $R$.

 Before stating our results, we also need to recall the \emph{Bass-Heller-Swan decomposition} in $K$-theory, which permits to {decompose} the $K$-theory of $R[\Z]$. For thorough approaches to algebraic $K$-theory the reader is referred to \cite{Bas68}, \cite{Wei13} or \cite{Luc21}.

\begin{thm}[\cite{Sri91}, Theorem 9.8]
\label{BHS}
Given a ring $R$ and $q\in\Z$, there exists an isomorphism $$K_q(R[\Z])\simeq K_q(R)\oplus K_{q-1}(R)\oplus NK_q(R)\oplus NK_q(R),$$ which is natural in the ring $R$.
\end{thm}

{The additional terms $NK_q(R)$ are called the Nil-terms, and $NK_n(R)$ is defined as the kernel of the homomorphism induced in $K_q$ by the homomorphism $R[t]\rightarrow R$ which sends $t$ to $1$. These terms vanish for a regular ring $R$, see \cite[Section 9]{Sri91}.}

{In our computations it will be important to understand the endomorphism $ind_n$ of $K_q(R[\Z])$ induced by multiplication by $n$ in $\Z$. The references for the sequel are \cite[Section 2]{HaLu12}, \cite{Sti82} and \cite{Wei81}.}

According to Bass-Heller-Swan-decomposition, $ind_n$ can be seen as a homomorphism
$$ind_n\colon K_q(R)\oplus K_{q-1}(R)\oplus NK_q(R)\oplus NK_q(R)\rightarrow K_q(R)\oplus K_{q-1}(R)\oplus NK_q(R)\oplus NK_q(R).$$
\noindent{Here, by naturality of the decomposition, the image of $ind_n|_{K_q(R)}$ lies inside $K_q(R)$, the image of $ind_n|_{K_{q-1}(R)}$ lies inside $K_{q-1}(R)$ and analogously for the Nil-terms. Now, the restriction of  $ind_n$ to $K_q(R)$ is the identity and the restriction of  $ind_n$ to $K_{q-1}(R)$ is multiplication by $n$.  By Farrell's trick \cite{Far77}, $ind_n$ admits a transfer $res_n$ such that $res_n\circ ind_n$ is multiplication by $n$.
This implies the following:}

{\begin{prop} \label{kernel}
In the previous notation, the kernel of $ind_n$ is isomorphic to a direct sum $T_1(K_q(R))\oplus T_2(K_q(R))$, where $T_1(K_q(R))$ is the $n$-torsion subgroup of $K_{q-1}(R)$ and $T_2(K_q(R))$ is a subgroup of the  $n$-torsion subgroup of $NK_q(R)\oplus NK_q(R).$
\end{prop}}

{It should be remarked that the restriction of $ind_n|_{NK_q(R)}$ (sometimes call the \emph{Frobenius map} in the literature) is related to the action of big Witt vectors in $NK_q(R)$, and is hard to describe in general. Anyhow, the previous proposition will be important when computing the Bredon homology of Artin groups of dihedral type.}

The following  propositions
{particularize} our previous results when the coefficient module is $K$-theory.

 For the sake of clarity, we maintain the notation from Proposition \ref{mainMV}, although we are aware that sometimes can be found redundant.

\begin{prop}
\label{Anordinary}

Let us consider an Artin group $A_n$ of dihedral type, $q\in\mathbb{Z}$, $R$ a ring. Then we have:

\begin{itemize}

\item $H_0^{Fin}(A_n,K_q(R[-]))=K_q(R)$.

\item $H_1^{Fin}(A_n,K_q(R[-]))=K_q(R)\oplus K_q(R)$ if $n$ is even

\item $H_1^{Fin}(A_n,K_q(R[-]))=K_q(R)$ if $n$ is odd.

\item $H_2^{Fin}(A_n,K_q(R[-]))=K_q(R)$ if $n$ is even
\item $H_2^{Fin}(A_n,K_q(R[-]))=0$ if $n$ is odd.

\item $H_i^{Fin}(A_n,K_q(R[-]))=0$ if $i\geq 3$.

\end{itemize}

\end{prop}

\begin{proof}

As $A_n$ is torsion-free, Bredon homology with respect to the family of finite groups is in fact ordinary homology. The proposition is then obtained by applying Universal Coefficient Theorem \cite[Theorem 3A.3]{Hat02} to the ordinary homology groups of $A_n$ (Proposition \ref{OrdHom}), taking into account that the latter are torsion-free and that the $K$-theory groups are abelian.
\end{proof}

We consider now {the} Bredon homology of the commensurators.

\begin{prop}
\label{FinComm}

Let us consider an Artin group $A_n$ of dihedral type, $q\in\mathbb{Z}$, $R$ a ring. Let $\Comm_{A_n}[H]$ be the commensurator of a virtually cyclic group in $A_n$. Then:

\begin{itemize}

\item If $\Comm_{A_n}[H]=\mathbb{Z}\oplus\mathbb{Z}$, then

$$H_i^{Fin\cap \Comm_{A_n}[H]}(\Comm_{A_n}[H],K_q(R[-]))=\begin{cases} K_q(R) & \text{ if $i=0,2$}\\
K_q(R)\oplus K_q(R) & \text{ if i =1}\\
0 & \text{otherwise}. \end{cases}$$

\item If $\Comm_{A_n}[H]=A_n$, then $$H_i^{Fin\cap \Comm_{A_n}[H]}(\Comm_{A_n}[H],K_q(R[-]))=H_i^{Fin}(A_n,K_q(R))$$ and this case was described in the previous proposition.

\end{itemize}

\end{prop}

\begin{proof}

Given a commensurator $\textrm{Comm}_{A_n}[H]$, its Bredon homology with respect to the family $Fin\cap \textrm{Comm}_{A_n}[H]$ is its ordinary homology. Then, {as in the previous proposition}, the result follows by applying the Universal Coefficient Theorem to the ordinary homology groups of the commensurators.
\end{proof}

We now undertake the remaining case. In item (3) we follow the notation of Proposition \ref{kernel}.
Moreover, we use the following notation.
\begin{nt}\label{not: CKR}
For $n$ odd, we denote by $C(K_q(R))$ the cokernel of the homomorphism $$\tilde{f}\colon K_q(R[\Z])\to K_q(R[\Z])\oplus K_q(R[\Z])$$ induced by
$f\colon R[\Z]\rightarrow R[\Z]\oplus  R[\Z]$ induced in the first component by multiplication by 2 and in the second by multiplication by $n=2k+1$.
\end{nt}
By Bass-Heller-Swan decomposition, $C(K_q(R))$ this is a quotient of $K_q(R)\oplus K_q(R)\bigoplus(\oplus_{i=1}^2 K_{q-1}(R))\bigoplus (\oplus_{i=1}^4 NK_q(R))$. Moreover, $\tilde{f}$ restricts to the component corresponding to  $K_q(R)$ of the Bass-Heller-Swan decomposition of $K_q(R[\Z])$ to the diagonal map to $K_q(R)\oplus K_q(R)$, a component  of $K_q(R[\Z])\oplus K_q(R[\Z])$.
Thus, $C(K_q(R))$ can be viewed as a quotient of $K_q(R)\bigoplus(\oplus_{i=1}^2 K_{q-1}(R))\bigoplus (\oplus_{i=1}^4NK_q(R))$ that can be identified in many cases, see Section \ref{Sect:concrete}.

\begin{prop}
\label{FHcomm}
Let $A_n$  be an Artin group of dihedral type, $q\in\mathbb{Z}$, $R$ a ring.
Consider for every non-trivial virtually cyclic $H\leqslant A_n$ and the family $\mathcal{F}[H]$. Then:

\begin{enumerate}

\item If $i\geq 2$, then $$H_i^{\mathcal{F}[H]}(\Comm_{A_n}[H],K_q(R[-]))=0.$$

\item If $H\cap Z(A_n)$ is trivial, then $$H_0^{\mathcal{F}[H]}(\Comm_{A_n}[H],K_q(R[-]))=H_1^{\mathcal{F}[H]}(\Comm_{A_n}[H],K_q(R[-]))=K_q(R[\mathbb{Z}]).$$

\item If $n$ is odd and $\Comm_{A_n}[H]=A_n$, then $$H_0^{\mathcal{F}[H]}(\Comm_{A_n}[H],K_q(R[-]))=C(K_q(R))$$ and $$H_1^{\mathcal{F}[H]}(\Comm_{A_n}[H],K_q(R[-]))=T_1(K_q(R))\oplus T_2(K_q(R)).$$

\item If $n$ is even and $\Comm_{A_n}[H]=A_n$, then $$H_0^{\mathcal{F}[H]}(\Comm_{A_n}[H],K_q(R[-]))=H_1^{\mathcal{F}[H]}(\Comm_{A_n}[H],K_q(R[-]))=K_q(R[\mathbb{Z}]).$$

\end{enumerate}

\end{prop}

\begin{proof}

We will check every item separately.

\begin{enumerate}

\item It follows straightforward from Corollary \ref{zerohomology}.

\item The claim follows from item (i) in Proposition \ref{homology0and1}, taking into account that $H$ is infinite cyclic an then $K_q(R[H])=K_q(R[\mathbb{Z}])$.

\item Let $n=2k+1$. Consider the homomorphism $f\colon M(A_{n}/Z(A_n))\to M(A_n/\langle b(ab)^k\rangle ) \oplus M(A_n/\langle ab \rangle)$ defined in item (2) of Proposition \ref{homology0and1}.
Taking $M(A_n/-)=K_q(R[-])$, we obtain a homomorphism $f_K:K_q(R[Z(A_n)])\to K_q(R[\langle b(ab)^k\rangle])\oplus K_q(R[\langle ab \rangle])$.
Observe that the two components of $f_K$ are respectively induced by the inclusions $Z(A_n)\hookrightarrow \langle b(ab)^k\rangle$ and $Z(A_n)\hookrightarrow \langle ab \rangle$, which are both inclusions $\Z\hookrightarrow \Z$ given respectively by multiplication by 2 and multiplication by $n$.
Then, $f_K$ can be seen as the homomorphism $K_q(R[\Z])\rightarrow K_q(R[\Z])\oplus K_q(R[\Z])$ induced by each multiplication in the corresponding component.
According to Proposition \ref{kernel}, the kernel of the first component of this homomorphism is equal to $T_1(K_q(R))\oplus T_2(K_q(R))$, while its cokernel is $C(K_q(R))$ by definition.
The result now follows from item (2) in Proposition \ref{homology0and1}.

\item  In this case $A_n=\Comm_{A_n}[H]$. The claim follows from item (3) in Proposition \ref{homology0and1}, taking into account that $H$ is infinite cyclic and  then $K_q(R[H])=K_q(R[\mathbb{Z}])$.
\end{enumerate}
\end{proof}

\subsection{Understanding the homomorphisms of Proposition \ref{mainMV}}
\label{morphism}
In our way to describe the Bredon homology of $A_n$ with respect to the family of virtually cyclic groups, we should describe to some extent the homomorphisms that appear in the Mayer-Vietoris sequence of Proposition \ref{mainMV}.

 In the following we will use without explicit mention the previous three propositions, which identify the terms of the exact sequence. We also maintain the name of the homomorphisms in the sequence.
Note that the superscript of the homomorphism $g^k_i$ specifies the degree of the homology in the source of $g^k_i$.
Moreover,  when we need to refer to the $j$-th component of $g_i^k$, we will write $g_{ij}^k$.
For example, $g_{22}^1$ is the second component of the homomorphism $g_2^1$ defined over the first homology group.

To prove our results, we need to analyze in detail the following homomorphism
\begin{equation}\label{g2}
\bigoplus_{[H]}H_i^{Fin\cap \Comm_{A_n}[H]}(\Comm_{A_n}[H],M) \stackrel{g_2^i}{\rightarrow} \end{equation}
$$ \stackrel{g_2^i}{\rightarrow}\left(\bigoplus_{[H]}H_i^{\mathcal{F}[H]}(\Comm_{A_n}[H],M)\right)\oplus H_i^{Fin}(A_n,M).
$$

The homomorphism $g_{21}^i$ is induced by  the  vertical left arrow $\sqcup_{[H]\in I}\mathrm{id}_{A_n}\times_{\Comm_{A_n}[H]}f_{[H]}$ of L\"{u}ck-Weiermann push-out  \eqref{eq:Luck-Weiermann} below (see also Theorem \ref{maintheorem}).
In turn, the homomorphism $g_{22}^i$ is induced by the inclusion in the upper horizontal arrow of the push-out  \eqref{eq:Luck-Weiermann}
\begin{equation}
\label{eq:Luck-Weiermann}
\xymatrix{  \coprod_{[H]\in I}A_n\times_{\Comm_{A_n}[H]}\uE \Comm_{A_n}[H] \ar[r]^{\hspace{2cm} i} \ar[d]^{\coprod_{[H]\in I}id_{A_n}\times_{\Comm_{A_n}[H]}f_{[H]}}  & \uE A_n \ar[d] \\
\coprod_{[H]\in I}A_n\times_{\Comm_{A_n}[H]}E_{\F [H]}\Comm_{A_n}[H]  \ar[r] & X. }
\end{equation}

We will decompose the homomorphism $g_2^i$ of Equation \eqref{g2} into homomorphisms $g_{2[H]}^i$ which are the restriction of $g_2^i$ to the factor of the domain corresponding to $[H]$. We will write

\begin{equation}
\label{eq:g}
g_{2[H]}^i\colon H_i^{Fin\cap \Comm_{A_n}[H]}(\Comm_{A_n}[H],M)\to H_i^{\mathcal{F}[H]}(\Comm_{A_n}[H],M)\oplus H_i^{Fin}(A_n,M).
\end{equation}
Note that we are making a slight abuse of notation, as the real codomain of $g_{2[H]}^i$ is the same as {that of} $g_2^i$, but we have chosen to write only the subgroup where the image of $g_{2[H]}^i$ lies.
Moreover, when needed, we will further decompose $g^i_{2[H]}$ into $g^i_{21[H]}\oplus g^i_{22[H]}$ indicating the different factors of the image of $g^i_{2[H]}$. Observe that this notation is coherent with the previous one.

As all the commensurators are torsion-free, $H_i^{Fin\cap \Comm_{A_n}[H]}(\Comm_{A_n}[H],M)$ is ordinary homology with coefficients in $M(\Comm_{A_n}[H]/\{1\})$.


We now analyze the homomorphism $g_2^i$ on the case $M=K_q(R[-])$. Before that, we introduce the following notation.
\begin{nt}
We will denote
$N_q^{[H]} =K_{q-1}(R)\oplus NK_q(R)\oplus NK_q(R)$, where the superindex means that this group is associated to a concrete commensurability class $[H]$.
\end{nt}
We now will give more detailed descriptions of the homomorphism $g_2^i$ and the cokernels introduced above.

\subsubsection{{\bf The homomorphism $g_2^2$ when $n$ odd:}}
\label{NqH}
In this case we have that the codomain of $g_2^2$ is  $(\bigoplus H_2^{\mathcal{F}[H]}(\Comm_{A_n}[H],M))\oplus H_2^{Fin}(A_n,M)=\{0\}$, and therefore for all commensurability classes of infinite  cyclic subgroups $[H]$ the homomorphisms $g_{2[H]}^2$ are trivial.

\subsubsection{{\bf The homomorphism $g_2^2$ when $n$ even:}}
In this case we have that the codomain of $g_2^2$ is  $(\bigoplus H_2^{\mathcal{F}[H]}(\Comm_{A_n}[H],M))\oplus H_2^{Fin}(A_n,M)=(\bigoplus_{[H]}\{0\})\oplus K_q(R)$.
Moreover, by Proposition \ref{Anordinary} and Proposition \ref{FinComm},
the domain of all $g_{2[H]}^2$ are the same, {and} we have
$$g_{2[H]}^2 \colon K_q(R)\to \{0\}\oplus K_q(R).$$
Note that $g_{22[H]}^2$ is induced by the inclusion given by the upper arrow in the push-out.
When $H=Z(A_n)$, we have that $\Comm_{A_n}[H]=A_n$ and therefore this inclusion is the identity.
Since $g_{22[Z(A_n)]}^2$ is surjective, we have that $g_2^2$ is surjective.

\subsubsection{{\bf The homomorphism $g_2^1$ when $n$ is odd:}}
\label{g21}
{Before we describe this case, we need to make some considerations.
Recall that we assume that $H$ is normal in $\textrm{Comm}_{A_n}[H]$.
Observe that in the models described in Section \ref{HiF[H]} for $\underline{E}(\textrm{Comm}_{A_n} [H]/H)$, the stabilizers of the edges are trivial for every $H$.
This means that if we consider the space $\underline{E}(\textrm{Comm}_{A_n} [H]/H)$ as a model for $E_{\F [H]}\Comm_{A_n}[H]$ (see Proposition \ref{induced}) the stabilizers of the edges are always isomorphic to $H$.
On the other hand, in any model of $\underline{E}\textrm{Comm}_{A_n} [H]$ the stabilizer of the edges should be trivial, as the action of $\textrm{Comm}_{A_n} [H]$ is free.
Then, the homomorphism induced by $f_{[H]}:\underline{E}\textrm{Comm}_{A_n} [H]\rightarrow E_{\F [H]}\Comm_{A_n}[H]$ in the first chain group of the Bredon complex with coefficients in a module $M$ takes every copy of $M(\textrm{Comm}_{A_n} [H]/{1})$ to a copy of $M(\textrm{Comm}_{A_n}[H]/H)$ with the homomorphism induced by the inclusion of the trivial group in $H$. In particular, if $M=K_q(R[-])$ for some $q$, the corresponding homomorphism $K_q(R)\rightarrow K_q(R[\mathbb{Z}])$ is given by the inclusion of $K_q(R)$ in the corresponding piece of the Bass-Heller-Swan decomposition of $K_q(R[\mathbb{Z}])$. This fact will be very useful in the sequel}.

In the following we use the notation of Proposition \ref{kernel} when needed.
We have that $$g_{2[Z(A_n)]}^1\colon K_q(R)\to T_1(K_q(R))\oplus T_2(K_q(R))\oplus K_q(R)$$ and for $H$ nontrivial, $[H]\neq [Z(A_n)]$
$$g_{2[H]}^1\colon  K_q(R)\oplus K_q(R)\to   K_q(R[\Z]) \oplus K_q(R).$$

Note that there are infinitely many commensurability classes of infinity cyclic subgroups different from the $[Z(A_n)]$ and hence infinitely many terms of this  kind.
We now examine the cases $[H]\neq [Z(A_n)]$ and $[H]=[Z(A_n)]$ separately.

If $[H]=[Z(A_n)]$ then we have $\Comm_{A_n}[H]=A_n$.

 In order to describe $g^{1}_{21[Z(A_n)]}$, consider the homomorphism $$C_1^{Fin}(\underline{E}\textrm{Comm}_{A_n}[H],K_q)\rightarrow C_1^{\F [H]}(E_{\F [H]}\Comm_{A_n}[H],K_q)$$ at the level of Bredon chains that induces $g^{1}_{21[Z(A_n)]}$ in homology. According to the previous considerations, this homomorphism is given by the inclusion $i:K_q(R)\hookrightarrow K_q(R[H])=K_q(R[\Z])$ given by Bass-Heller-Swan decomposition. But by item (3) of Proposition \ref{FHcomm}, the image of this homomorphism is trivial in $H_1^{\F [H]}(E_{\F [H]}\Comm_{A_n}[H],K_q)$, and hence $g^{1}_{21[Z(A_n)]}$ is also trivial. In turn, $g^{1}_{22[Z(A_n)]}$ is the identity.

If $[H]\neq[Z(A_n)]$ then $\Comm_{A_n}[H]=\Z^2$, and $\Comm_{A_n}[H]/H$ is isomorphic to $\Z$.
In this case $g_{21 [H]}^1$ is given by {a} homomorphism $K_q(R)\oplus K_q(R)\rightarrow K_q(R[\Z])$, where we assume that the first component of the domain corresponds to $H$ and the second to $Z(A_n)$. As the homomorphism $g_{21 [H]}^1$ is induced in homology by the quotient homomorphism $\textrm{Comm}_{A_n}[H]\rightarrow \textrm{Comm}_{A_n}[H]/H$, the previous results imply that the first component of $g_{21 [H]}^1$ is trivial, while the second, which corresponds to the center, identifies the copy of $K_q(R)$ in the Bass-Heller-Swan decomposition of $K_q(R[\Z])$.

 On the other hand,  $g_{22[H]}^1\colon (\Z\oplus\Z)\otimes K_q(R)\rightarrow \Z\otimes K_q(R)$  is defined by the abelianization of the inclusion $H_1(\Comm_{A_n}[H])\rightarrow H_1(A_n)$ in the first component of the tensor product and by the identity in the second.

Now since the image of $g^1_{2[Z(A_n)]}$ is precisely given by the copy of $K_q(R)$ that corresponds to $H_1(A_n,K_q)$, the previous computations imply that the cokernel of $g^1_2$ is then equal to $(\bigoplus_{[H]\neq [Z(A_n)]}N_q^{[H]})\oplus T_1(K_q(R))\oplus T_2(K_q(R))$.

\subsubsection{ {\bf The homomorphism $g_2^1$ when $n$ is even:}}

This homomomorphism is defined in equation \eqref{eq:g} and, according to Propositions {\ref{Anordinary}, \ref{FinComm} and \ref{FHcomm}}, its first component is given by
$$g_{2[H]}^1 \colon K_q(R)\oplus K_q(R)\to K_q(R[\Z]) \oplus ( K_q(R)\oplus K_q(R))$$
for every commensurability class $[H]$. Let us describe this component with more detail.

We consider first the case $[H] \neq [Z(A_n)]$. Here, the same argument as in the odd case proves that the component of $g_{21[H]}^1$ given by the center is the inclusion $K_q(R)\hookrightarrow K_q(R[\Z])$ via Bass-Heller-Swan decomposition, and the other component is trivial.

Also when $[H]\neq [Z(A_n)]$ the homomorphism $g_{22[H]}^1$ is identified (via Proposition \ref{FinComm} and the Universal Coefficient Theorem) with a homomorphism $(\Z\oplus \Z) \otimes K_q(R)\rightarrow (\Z\oplus\Z)\otimes K_q(R)$, which comes, as above, from tensoring with $K_q(R)$ the homomomorphism $H_1(\Comm_{A_n}[H])\rightarrow H_1(A_n)$ given by abelianization of the inclusion of the commensurator in $A_n$.

Now we consider $[H]=[Z(A_n)]$. As in the previous case, the homomorphism $g_{21[H]}^1$ can be described as:
$$H_1^{Fin}(A_n,K_q(-))\to H_1^{\F [H]}(A_n,\pi^{-1} K_q(-)) \cong H_1^{Fin}(A_n/Z(A_n),\pi^{-1} K_q(-)).$$
Now recall from Proposition \ref{homology0and1} that
$$H_1^{Fin}(A_n/Z(A_n),\pi^{-1} K_q(-))=\pi^{-1}K_q(R[\{1\}])=K_q(R[Z(A_n)])=K_q(R[\Z]),$$
where $Z(A_n)$ is interpreted here as the stabilizer of the unique $A_n$-class of edges in the model of $E_{\F [H]}A_n$ described in Section \ref{HiF[H]}.
On the other hand, the two copies of $K_q(R)$ in $H_1(A_n,K_q(R))$ come from taking values of the module $K_q(R[-])$ on the trivial group, interpreted as the stabilizer of two different $A_n$-classes of edges in a model of $EA_n$. Again taking into account our previous considerations about stabilizers, the two components of $g_{21[Z(A_n)]}^1 \colon K_q(R)\oplus K_q(R)\to K_q(R[\Z])$ induce inclusion of the $K_0(R)$ via Bass-Heller-Swan decomposition. On the other hand, it is clear that $g_{22[Z(A_n)]}^1$ is the identity.

As $H_1(A_n)$ is a free abelian group of rank 2 generated by the images of $a$ and $b$ under abelianization, we may assume that the two copies of $K_q(R)$ in the image of $g_{2[H]}^1$ correspond respectively to these two copies of $\Z$, after tensoring with $K_q(R)$. Now observe that if $H=\langle a\rangle$, the image of the restriction of  $g_{22[H]}^1$ to the homology of its commensurator is exactly the first of the two copies of $K_q(R)$, while the image of the restriction of $g_{22[H]}^1$ to the homology of the commensurator of $H=\langle b\rangle$ is the other copy. As the restrictions of $g_{21[H]}^1$ to the homology of these commensurators are trivial, we obtain that $\{0\}\oplus K_q(R)\oplus K_q(R)$ lies in the image of $g_{2[H]}^1$. In fact, the description of $g_{21[H]}^1$ for $H=Z(A_n)$ implies that the copy of $K_q(R)$ inside $K_q(R[\Z])$ is also in the image, and now it is easy to conclude that the image of $g_{2[H]}^1$ is in fact $(\bigoplus_{[H]}) K_q(R))\oplus K_q(R)\oplus K_q(R)$, corresponding the big direct sum to the copies of $K_q(R)$ included in each copy of $K_q(R[\Z])$, and the remaining two copies corresponding to the ordinary homology of $A_n$. In particular, we have that $\textrm{coker } g_2^1=\bigoplus_{[H]} N_q^{[H]}.$


\subsubsection{{\bf The homomorphism $g_2^0$ for every $n$:}}

Similar considerations to those of the beginning of Section \ref{g21} hold here.
If $x$ is a vertex of $\underline{E}\textrm{Comm}_{A_n}[H]$ such that the stabilizer of $f_{[H]}(x)$ is infinite cyclic (here $f_{[H]}$ is the function of \eqref{eq:Luck-Weiermann}), then the induced map $K_q(R)\rightarrow K_q(R[\Z])$ induces the injection in the correspondent component of Bass-Heller-Swan decomposition. This is always the case except when $n$ is odd, $H=Z(A_n)$ and $f_{[H]}(x)$ has the shape $gC_{\infty}$ if we consider $\underline{E}(\textrm{Comm}_{A_n} [H]/H)$ as a $\textrm{Comm}_{A_n} [H]/H$-complex.
This will be enough to describe $g_2^0$ to the extent we need.

We begin describing $g_{2[H]}^0$ according to the different prossibilities of $[H]$ and $n$.

When $H\neq Z(A_n)$, the results of Propositions \ref{Anordinary}, \ref{FinComm} and \ref{FHcomm} imply that $g_{2[H]}^0$ is defined in the following way:
$$g_{2[H]}^0\colon K_q(R)\to K_q(R[\Z])\oplus K_q(R).$$
By the previous considerations, the homomorphism $g_{21[H]}^0$ identifies $K_q(R)$ as the corresponding direct summand of $K_q(R[\mathbb{Z}])$ in the Bass-Heller-Swan decomposition, while $g^{0}_{22[H]}$ is induced by the inclusion $\textrm{Comm}_{A_n} [H]\hookrightarrow A_n$, counts the number of connected components of the classifying space, and then is the identity.

When $n$ is even and $H=Z(A_n)$, we have that $H_0^{\F [H]}(\textrm{Comm}_{A_n} [H],K_q(-))=K_q(R[\Z])$ by Proposition \ref{FHcomm}, and this copy of $\Z$ corresponds to a stabilizer which is isomorphic to $C_n$, a cyclic group of order $n$. Then, by the previous considerations, the homomorphism $g_{2[H]}^0$ behaves as in the case $n$ even and $H\neq Z(A_n)$.

In the case $n$ odd and $H=Z(A_n)$, again by Proposition \ref{FHcomm} we have that $ H_0^{\F [H]}(\textrm{Comm}_{A_n} [H],K_q(-))=C(K_q(R))$, where $C(K_q(R))$ was defined in Notation \ref{not: CKR}.
Then $g_{21[H]}^0$ is defined as a homomorphism
$$g_{2[H]}^0\colon K_q(R)\to C(K_q(R))\oplus K_q(R).$$
Remark from the definition of $C(K_q(R))$ that this group is the direct sum of $K_q(R)$ with quotients of $\oplus_{i=1}^2 K_{q-1}(R)$ and $\oplus_{i=1}^4 NK_q(R)$,
and that this copy of $K_q(R)$ comes from the identification of the respective copies of $K_q(R)$ that appear in the Bass-Heller-Swan decomposition of $K_q(R[\langle b(ab)^k\rangle])$ and $K_q(R[\langle ab\rangle])$.
These copies correspond respectively to the stabilizers $C_2$ and $C_{2k+1}$ in $\underline{E}(\textrm{Comm}_{A_n} [H]/H)$ (as a model for the $\textrm{Comm}_{A_n} [H]/H$-action). Then, the previous considerations about stabilizers imply that $g_{21[H]}^0$ maps this $K_0(R)$ isomorphically to the mentioned copy of itself inside $C(K_q(R))$.
In turn, $g_{22[H]}^0$ is again the identity, arguing as in the even case.

With the previous information, we can describe $\mathrm{coker}g_2^0$.

Now if $n$ is even, observe that the source of $g_2^0$ is $\oplus_{[H]} K_q(R)$ and the codomain is $(\oplus_{[H]}K_q(R[\Z]))\oplus K_q(R)$.
Moreover, the image of every $g_{2[H]}^0$ consists in a copy of $K_q(R)$ inside $K_q(R[\Z])$ (again the copy that appears in the Bass-Heller-Swan decomposition), and the copy of $K_q(R)$ that corresponds to $H_0^{Fin}(A_n,M)$, which is fixed for any choice of $H$. Hence, the cokernel of the homomorphism $g_2^0$ is isomorphic to the quotient of a direct sum of copies of $K_q(R[\Z])$ (indexed by $[H]$) by the identification of all the copies $K_q(R)$ which are the images of the homomorphisms $g^0_{22[H]}$. Observe that  the copies of $K_q(R[\Z])$ are ``glued" by the copy of $K_q(R)$ that corresponds to $H_0^{Fin}(A_n,M)$, and hence we have $$\mathrm{ coker } g_2^0=(\bigoplus_{[H]} N_q^{[H]})\oplus K_q(R).$$

When $n$ is odd, we only need to take into account that in the case of $H=Z(A_n)$ the role of $K_q(R[\Z])$ in the codomain is played by $C(K_q(R))$.
Then, using the same argument as in the previous case and denoting by $\overline{C}(K_q(R))$ the quotient of $C(K_q(R))$ under the copy of $K_q(R)$ in the Bass-Heller-Swan decomposition, we obtain that: $$\mathrm{ coker } g_2^0=(\bigoplus_{[H]\neq [Z(A_n)]} N_q^{[H]})\oplus K_q(R)\oplus\overline{C}(K_q(R)).$$

Observe that $g_2^0$ is a monomorphism for every $n$, as the inclusions of $K_q(R)$ in $K_q(R[\Z])$ and $C(K_q(R))$ are so.

\vspace{0.5cm}

Now we can describe the Bredon homology of $A_n$ with respect to the family of virtually cyclic groups.
Observe that the (co)kernels of the statement have been previously described, and that the groups $N_q^{[H]}$ were defined in Section \ref{NqH}.

\begin{thm}
\label{Thm:Bredon}
Let $A_n$ be an Artin group of dihedral type. In the previous notation, we have the following:
\begin{enumerate}
  \setlength{\itemindent}{-2em}
\item $H_i^{vc}(A_n,K_q (R[-]))=\{0\}$ for $i\geq 4$.
\item $H_3^{vc}(A_n,K_q (R[-]))=\begin{cases}\bigoplus_{[H]\neq [Z(A_n)]} K_q(R) & \text{$n$ odd}\\
\ker g_2^2 & \text{$n$ even.}
\end{cases}$

\item $H_2^{vc}(A_n,K_q (R[-]))= \ker g_2^1$.

\item $H_1^{vc}(A_n,K_q (R[-]))= \emph{coker } g_2^1= \begin{cases}(\bigoplus_{[H]\neq [Z(A_n)]}N_q^{[H]})\oplus T_1(K_q(R))\oplus T_2(K_q(R)) & \text{$n$ odd}\\
\bigoplus_{[H]} N_q^{[H]}
  & \text{$n$ even.}
\end{cases}$%
\item $H_0^{vc}(A_n,K_q (R[-]))= \emph{ coker } g_2^0= \begin{cases}(\bigoplus_{[H]\neq [Z(A_n)]} N_q^{[H]})\oplus K_q(R)\oplus\overline{C}(K_qR) & \text{$n$ odd}\\
(\bigoplus_{[H]} N_q^{[H]})\oplus K_q(R)
  & \text{$n$ even.}
  \end{cases}$%
\end{enumerate}
\end{thm}

\begin{proof}

The proof is based on the sequence of Proposition \ref{mainMV} and the previous homological computations. We will check the five items in a separate way.

\begin{enumerate}
\item By Proposition \ref{Anordinary}, Proposition \ref{FinComm}  and Proposition \ref{FHcomm} the sequence of Proposition \ref{mainMV} is identically trivial to the left of $g_3^3$.

\item As the exact sequence of Proposition \ref{mainMV} is identically trivial to the left of $g_3^3$, the map $g_3^3$ is trivial and by exactness, $H_3^{vc}(A_n, K_q(R[-]))$ is isomorphic to the kernel of $g_2^2$. This completes the even case.

For the case $n$ odd, we claim that $g_1^3$ is an isomorphism.
Indeed, from Proposition \ref{Anordinary} and Proposition \ref{FHcomm}, the term $(\bigoplus H_2^{\mathcal{F}[H]}(\Comm_{A_n}[H],M))\oplus H_2^{Fin}(A_n,M)=\{0\}$ and in particular  $\textrm{Im}(g_2^2)=\{0\}$.
Since $g_1^3$ is injective, the claim follows.
By Proposition \ref{FinComm}, the specific description of the odd case follows.

\item There are two different arguments depending if $n$ is odd or even.

For $n$ odd, as the term  $(\bigoplus H_2^{\mathcal{F}[H]}(\Comm_{A_n}[H],M))\oplus H_2^{Fin}(A_n,M)$  is trivial, the statement is a direct consequence of the exactness of the sequence of Proposition \ref{mainMV}.

For $n$ even, we have seen in the discussion above about $g_2^2$ that this map is surjective, and hence $g_3^2$ is the trivial map.
This implies that $H_2^{vc}(A_n, K_q(R[-]))$ is the image of $g_1^2$, or equivalently, the kernel of $g_2^1$.

\item The previous description of the homomorphisms in the Mayer-Vietoris sequence proves that $g_2^0$ is a monomorphism, and hence $g_3^1$ is surjective and $H_1^{vc}(A_n,K_q( R[-]))$ the cokernel of $g_2^1$, which has been described above in terms of the summands $N_q^{[H]}$.

\item Since the Mayer-Vietoris sequence {ends} at $H_0^{vc}(A_n ,K_q( R[-]))$ this term is equal to the image of $g_3^0$ which is equal to the cokernel of $g_2^0$ that was  described above.
\end{enumerate}
\end{proof}

As said in the introduction, this theorem opens the door to concrete computations of Bredon homology of $A_n$ with respect to the family of virtually cyclic groups, provided there is available information about the $K$-theory of the coefficient ring. In particular, when $R$ is a regular ring the absence of Nil-terms and negative K-theory groups make the calculations easier and more precise.
For instance, we have
\begin{cor}\label{cor:K_0}
Let $A_n$, $n>2$ be an Artin group of dihedral type. Let $R$ be a regular ring. Then $K_0(RA_n)=K_0(R)$.
\end{cor}
\begin{proof}
As $R$ is regular, $K_{i}(R)$ vanishes for $i<0$ and also the Nil-Terms of the Bass-Heller-Swan decomposition.
In particular, by Theorem \ref{Thm:Bredon}, $H_0^{vc}(A_n, K_0(R[-]))= K_0(R)$.
Moreover, as $K_{i}(R)$ vanishes for $i<0$ we see  by Theorem \ref{Thm:Bredon}, that $H_j^{vc}(A_n,K_i(R[-]))=\{0\}$ for $i<0$.
Then the $E_2$-page of the Atiyah-Hirzebruch spectral sequence is concentrated in the non-negative part of the 0th, 1st, 2nd and 3rd columns, and $E_{\infty}^{0,0}=E_2^{0,0}= K_0(R)$.
\end{proof}

In particular, this implies that every finitely dominated $CW$-complex whose fundamental group is $A_n$ has the homotopy type of a finite $CW$-complex.

In the following section we compute these Bredon homology groups for different choices of the ring, including some non-regular ones.

\section{Computations of $H_i^{vc}(A_n,K_q (R[-]))$ for several coefficient rings}
\label{Sect:concrete}

In this section we use Theorem \ref{Thm:Bredon} to describe $H_i^{vc}(A_n,K_q (R[-]))$ for some instances of the ring $R$, both regular ($\mathbb{Z}, \mathbb{F}_q$) and non-regular ($\mathbb{Z}[\mathbb{Z}/{\bf 2}], \mathbb{Z}[\mathbb{Z}/{\bf 2}\times \mathbb{Z}/{\bf 2}]$ and $\mathbb{Z}[\mathbb{Z}/{\bf 4}]$).
We recall that all these groups give information about the $E^2$-term of the corresponding Atiyah spectral sequence. We point out that in the regular cases many groups can be computed, because the $K$-theory of $\Z$ is nearly known and the $K$-theory of $\mathbb{F}_q$ is known (see the corresponding examples below).
In the non-regular framework, by contrast, it is very difficult to find concrete descriptions of $K_q(R)$ for $q>1$, or of the corresponding Nil-terms.

Recall that the groups $H_i^{vc}(A_n,K_0 (R[-]))$ are trivial for $i\geq 4$ and any ring $R$.

\begin{ex}
\label{KZ}

First we compute the Bredon homology of $A_n$ with respect to the family of virtually cyclic subgroups, taking as coefficients $K_q(\Z [-])$, for $q=0,1,2$.

We need in our computations the lower algebraic $K$-theory groups of the integers. The groups $K_q(\Z)$ are known for all $q\leq 7$ and all $q\geq 8$ such that $q \not\equiv 0 \mod 4$,  see  \cite[page 2]{Wei05}.
For example,  the first values of $K_q(\Z)$ are given in the following table:
\begin{center}\begin{tabular}{c|c|c|c|c|c|c|c|c|c}
\hline
$q$ & $<0$ &0 & 1 & 2 & 3 & 4 & 5 & 6 & 7 \\
\hline
$K_q(\Z)$ & 0 & $\Z$ & $\Z/{\bf 2}$ & $\Z/{\bf 2}$ & $\Z/{\bf 48}$ & 0 & $\Z$ & 0 & $\Z/{\bf 240}$ \\
\hline
\end{tabular}
\end{center}

Observe that, as $\Z$ is regular, the Nil-terms in the Bass-Heller-Swan decomposition are trivial, and then for every $i\in \mathbb{N}$, $K_i(\Z[\Z])=K_i(\Z)\oplus K_{i-1}(\Z)$ and $N_q^{[H]}=K_{q-1}(\Z)$. Moreover, also by regularity, $K_i(\Z)=0$ if $i\leq -1$.

Taking into account of all these considerations, the previous theorem implies the following.

For $q=0$, we have:

\begin{enumerate}

\item $H_3^{vc}(A_n,K_0 (\Z[-]))\simeq H_2^{vc}(A_n,K_0 (\Z[-])) \simeq \bigoplus_{\aleph_0} \Z$.
\item {$H_1^{vc}(A_n,K_0 (\Z[-]))=0$.}
\item $H_0^{vc}(A_n,K_0 (\Z[-]))=\Z$.

\end{enumerate}

Now for $q=1$,

\begin{enumerate}

\item $H_3^{vc}(A_n,K_1 (\Z[-]))\simeq H_2^{vc}(A_n,K_1 (\Z[-])) \simeq \bigoplus_{\aleph_0} \Z /{\bf 2}$.
\item {$H_1^{vc}(A_n,K_1 (\Z[-]))=\bigoplus_{\aleph_0} \Z$}.
\item {$H_0^{vc}(A_n,K_1 (\Z[-]))=(\bigoplus_{\aleph_0} \Z)\oplus \Z /{\bf 2}$}.

\end{enumerate}

And finally, for $q=2$, all these groups are isomorphic to $\bigoplus_{\aleph_0} \Z /{\bf 2}$.

Let us briefly explain these computations.
We consider first $q=0$.
For $H_3^{vc}(A_n,K_0 (\Z[-]))$, the odd case is immediate.
In the even case, $H_3^{vc}(A_n,K_1 (\Z[-]))$ is the kernel of a homomorphism  $\bigoplus_{\aleph_0} \Z\rightarrow \Z$, and then isomorphic to $\bigoplus_{\aleph_0} \Z$.
This easy argument will be frequently used in the examples of this section without express mention.

We now deal with $H_2^{vc}(A_n,K_0 (\Z[-]))$.
We only deal with the odd case, the even one is very similar.
First observe that the kernel of $g_2^1$ is isomorphic to $\bigoplus_{\aleph_0} \Z$.
More precisely, taking account the description of Section \ref{g21} and the values of the $K_i(\Z)$, $g_2^1$ (for $n$ odd) is defined in the following way:

$$g_2^1: (\bigoplus_{[H]\neq [Z(A_n)]} \Z^2)\oplus \Z\rightarrow (\bigoplus_{[H]\neq [Z(A_n)]} \Z)\oplus \Z.$$

Here the kernel of $g_{21[H]}^1$ is isomorphic to $\Z$ for every $[H]\neq [Z(A_n)]$.
As there is an infinite number of such commensurators, we conclude that the kernel should be also infinite, and then $H_2^{vc}(A_n,K_1 (\Z[-]))=\bigoplus_{\aleph_0} \Z$.

Observe that, as $K_{-1}(\Z)=0$ and $\Z$ is regular, $N_0^{[H]}=T_i(K_0\Z)=\overline{C}(K_0\Z)=0$.
Now the values of $H_1^{vc}(A_n,K_0 (\Z[-]))$ and $H_0^{vc}(A_n,K_0 (\Z[-]))$ are easily deduced from items (4) and (5) of Theorem \ref{Thm:Bredon}.

Now take $q=1$. For $H_3^{vc}(A_n,K_1 (\Z[-]))$ and $H_2^{vc}(A_n,K_1 (\Z[-]))$ it is argued as in the previous case, taking into account that every subgroup of an $\mathbb{F}_2$-vector space is again an $\mathbb{F}_2$-vector space.
Now observe that $N_1^{[H]}=K_0(\Z)=\Z$ by the Bass-Heller-Swan decomposition, and regularity and the fact that $K_0(\Z)=\Z$ is torsion-free imply that $T_1(K_1\Z)=0$, $T_2(K_1\Z)=0$ and $\overline{C}(K_1\Z)=\Z$.
The values of $H_1^{vc}(A_n,K_1 (\Z[-]))$ and $H_0^{vc}(A_n,K_1 (\Z[-]))$ follow again from items (4) and (5) of Theorem \ref{Thm:Bredon}.

Finally, for $q=2$, the values of the homology are immediately implied by regularity and the fact that $K_2(\Z)=K_1(\Z)=\Z /{\bf 2}$.

\end{ex}


\begin{ex}

Now we will compute the groups $H_i^{vc}(A_n,K_q (\mathbb{F}_2[-]))$, for $0\leq q\leq 3$. First, the following table (see \cite{Qui73}) includes the algebraic $K$-groups that are necessary in our computations:

\begin{center}\begin{tabular}{c|c|c|c|c|c}
\hline
$q$ & $<0$ & 0 & 1 & 2 & 3 \\
\hline
$K_q(\mathbb{F}_2)$ & 0 & $\Z$ & 0 & 0 & $\Z/{\bf 3}$  \\
\hline
\end{tabular}
\end{center}

As $\mathbb{F}_2$ is regular, the same considerations about the Nil-terms and the negative $K$-groups apply also in this case. Hence, again as a direct consequence of Theorem \ref{Thm:Bredon} we have the following results. For $q=0$ and every $i\in\mathbb{N}$, $H_i^{vc}(A_n,K_0 (\mathbb{F}_2[-]))=H_i^{vc}(A_n,K_0 (\Z[-]))$, because $K_0(\Z)=K_0(\mathbb{F}_2)$ and then  $K_0 (\mathbb{F}_2[-])=K_0 (\Z[-])$ because both have the same Bass-Heller-Swan decomposition.

For $q=1$, we have $H_3^{vc}(A_n,K_1 (\mathbb{F}_2[-]))=H_2^{vc}(A_n,K_1 (\mathbb{F}_2[-]))=0$, as $K_1(\mathbb{F}_2)=0$. On the other hand, as {$N_1^{[H]}=\Z$} for every $H$, $H_1^{vc}(A_n,K_1 (\mathbb{F}_2[-]))=H_0^{vc}(A_n,K_1 (\mathbb{F}_2[-]))=\bigoplus_{\aleph_0}\Z$.

For $q=2$, the triviality of $K_2(\mathbb{F}_2)$ and $K_1(\mathbb{F}_2)$ implies the triviality of $H_i^{vc}(A_n,K_2 (\mathbb{F}_2[-]))$ for every $i$.

Finally, for $q=3$, we have $H_2^{vc}(A_n,K_3 (\mathbb{F}_2[-]))\simeq H_3^{vc}(A_n,K_3 (\mathbb{F}_2[-]))=\bigoplus_{\aleph_0}\Z /{\bf 3}$.
As the groups $N_3^{[H]}$ are  trivial and $K_2(\mathbb{F}_2)$ is so, we obtain that  $H_1^{vc}(A_n,K_3 (\mathbb{F}_2[-]))=\Z /{\bf 3}$ is trivial, and  $H_0^{vc}(A_n,K_3 (\mathbb{F}_2[-]))=\Z /{\bf 3}$.
\end{ex}

We follow with some non-regular examples, namely the rings $\Z [\Z /{\bf 2}]$, $\Z [\Z /{\bf 2}\times \Z /{\bf 2}]$ and $\Z [\Z /{\bf 4}]$, which we will respectively denote by $R_1$, $R_2$ and $R_3$.
We will compute the groups $H_i^{vc}(A_n,K_q (R_j[-]))$, for $0\leq i\leq 3$, $q=0,1$ and $1\leq j\leq 3$.
In order to do this we will need the values of their lower algebraic $K$-theory groups, as well as the Nil groups.
All these groups are displayed in the following table:

\begin{center}\begin{tabular}{c|c|c|c|c|c}
\hline
& $K_1$ & $K_0$ & $K_{-1}$ & $NK_0$ & $NK_1$ \\
\hline
$\Z [\Z /{\bf 2}]$  & $(\Z /{\bf 2})^2$ & $\Z$ & 0 & 0 & 0 \\
\hline
$\Z [\Z /{\bf 2}\times \Z /{\bf 2}]$  & $(\Z /{\bf 2})^3$ & $\Z$ & $\Z^r$ & $\bigoplus_{\aleph_0} \Z/{\bf 2}$ & $\bigoplus_{\aleph_0} \Z/{\bf 2} $\\
\hline
$\Z [\Z /{\bf 4}]$  & $\Z /{\bf 2}\times \Z /{\bf 4}$ & $\Z$ & $\Z^s$ & $\bigoplus_{\aleph_0} \Z/{\bf 2}$&$ \bigoplus_{\aleph_0} \Z/{\bf 2}$ \\
\hline
\end{tabular}
\end{center}

Let us briefly explain the values of the table.  By work of Oliver \cite[Theorem 14.1-2]{Oli88}, the kernels $SK_1(R_j)$ of the determinant maps are trivial, and hence (see for example \cite{Ste78}) $K_1(R_j)$ is the group of units of $R_j$. Now the values of $K_1$ follow from a theorem of Higman (\cite{Hig40}, see also \cite[II.4.1]{She78}). In turn, by \cite[Proposition 6]{Cas73}, the reduced $K_0$ of these three rings is trivial, and then \cite[Lemma 2.18]{Luc21} implies that $K_0(R_j)=\Z$ for every $j$. The values of the third column follow from work of Carter \cite[Theorem 1]{Car80}, being the figures $r$ and $s$ positive integers that depend on the Schur indexes. Finally, the Nil-terms of the two last columns were computed by Weibel in \cite{Wei09}.

We are now ready to describe the homology of the dihedral Artin groups with respect to the family of virtually cyclic groups, referred to the $K$-theory $K_q$ of these group rings, $q=0,1$. As before, our main tool is Theorem \ref{Thm:Bredon}.

\begin{ex}
We start with $R_1=\Z [\Z /{\bf 2}]$.
When $q=0$, $H^{vc}_3(A_n,K_0(R_1[-]))=H^{vc}_2(A_n,K_0(R_1[-]))=\bigoplus_{\aleph_0} \Z $.
Taking into account that $N_0^{[H]}=0$ for every $H$ (because $K_{-1}(R_1)$ is trivial, and also the Nil-terms), we obtain that $H^{vc}_1(A_n,K_0(R_1[-]))=0$.
On the other hand, $H^{vc}_0(A_n,K_0(R_1[-]))=\Z$ too.

If $q=1$, $H^{vc}_3(A_n,K_1(R_1[-]))=H^{vc}_2(A_n,K_1(R_1[-]))=\bigoplus_{\aleph_0} \Z /{\bf 2} $.
Moreover, as $\Z$ is torsion-free and then the correspondent $T_1(K_1(R_1))$ is trivial, we have $H^{vc}_1(A_n,K_1(R_1[-]))=\bigoplus_{\aleph_0}\Z$.
Finally, taking into account that $C(K_1(R_1))=\Z$ in this case, $H^{vc}_0(A_n,K_1(R_1[-]))=\Z /{\bf 2}\oplus \Z /{\bf 2}\oplus(\bigoplus_{\aleph_0} \Z)$.

\end{ex}

\begin{ex}
We continue by considering $R_2=\Z [\Z /{\bf 2}\times \Z /{\bf 2}]$.
Now $H^{vc}_3(A_n,K_0(R_2[-]))=H^{vc}_2(A_n,K_0(R_2[-]))=\bigoplus_{\aleph_0} \Z $, exactly as in the previous example. We have $N_0^{[H]}=K_{-1}(R_2)\oplus NK_0(R_2)\oplus NK_0(R_2)=\Z^r\oplus (\bigoplus_{\aleph_0}\Z /{\bf 2})$ and then $H^{vc}_1(A_n,K_0(R_2[-]))=(\bigoplus_{\aleph_0} \Z)\oplus (\bigoplus_{\aleph_0} \Z /{\bf 2})$.
 Observe that the extra term $T_2(K_0(R_2))$ is an $\mathbb{F}_2$-vector space of at most countable dimension, and then it is included in the previous direct sum.
 Finally, $H^{vc}_0(A_n,K_0(R_2[-]))=(\bigoplus_{\aleph_0} \Z)\oplus (\bigoplus_{\aleph_0} \Z /{\bf 2})$ as in the previous case, taking into account that $\overline{C}(K_0(R_2))$ is a direct sum of free abelian groups and an $\mathbb{F}_2$-vector space, both of at most countable dimension.

When $q=1$, $H^{vc}_3(A_n,K_1(R_2[-]))=H^{vc}_2(A_n,K_0(R_2[-]))=\bigoplus_{\aleph_0} \Z /{\bf 2}$, because $K_1(R_2)=(\Z /{\bf 2})^3$.
Now $N_q^{[H]}=\Z\oplus \bigoplus_{\aleph_0}\Z /{\bf 2}$, and then by similar reasons to the previous case, $H^{vc}_1(A_n,K_1(R_2[-]))=(\bigoplus_{\aleph_0} \Z)\oplus (\bigoplus_{\aleph_0} \Z /{\bf 2})$.
Similarly $H^{vc}_0(A_n,K_1(R_2[-]))=(\bigoplus_{\aleph_0} \Z)\oplus (\bigoplus_{\aleph_0} \Z /{\bf 2})$.

\end{ex}

\begin{ex}
To conclude, we consider $R_2=\Z [\Z /{\bf 4}]$.
As the groups $K_0$, $NK_0$ and $NK_1$ are all isomorphic as abelian groups to their counterparts in the previous example
and $K_{-1}$ is also free abelian, $H^{vc}_3(A_n,K_0(R_3[-]))=H^{vc}_2(A_n,K_0(R_3[-]))=\bigoplus_{\aleph_0} \Z $, $H^{vc}_1(A_n,K_0(R_3[-]))=(\bigoplus_{\aleph_0} \Z)\oplus (\bigoplus_{\aleph_0} \Z /{\bf 2})$, and $H^{vc}_0(A_n,K_0(R_3[-]))=(\bigoplus_{\aleph_0} \Z)\oplus (\bigoplus_{\aleph_0} \Z /{\bf 2})$.

For $q=1$, it is clear that
$H^{vc}_3(A_n,K_1(R_3[-]))=(\bigoplus_{\aleph_0} \Z /{\bf 2})\oplus (\bigoplus_{\aleph_0} \Z /{\bf 4})$.
Now, the analysis of $g_2^1$ in Section \ref{g21} guarantees that $H^{vc}_2(A_n,K_1(R_3[-]))$ should contain a subgroup isomorphic to $(\bigoplus_{\aleph_0} \Z /{\bf 2})\oplus (\bigoplus_{\aleph_0} \Z /{\bf 4})$, and then be isomorphic to it, as the source of $g_2^1$ is also isomorphic to $(\bigoplus_{\aleph_0} \Z /{\bf 2})\oplus (\bigoplus_{\aleph_0} \Z /{\bf 4})$.
We have $N_1^{[H]}=K_{0}(R_3)\oplus NK_0(R_2)\oplus NK_0(R_2)=\Z\oplus (\bigoplus_{\aleph_0}\Z /{\bf 2})$.
By an analogous reasoning to the case of $R_2$, $H^{vc}_1(A_n,K_1(R_3[-]))=(\bigoplus_{\aleph_0} \Z)\oplus (\bigoplus_{\aleph_0} \Z /{\bf 2})$.
Finally, $H^{vc}_0(A_n,K_1(R_3[-]))=(\bigoplus_{\aleph_0} \Z)\oplus (\bigoplus_{\aleph_0} \Z /{\bf 2})\oplus \Z /{\bf 4}$, corresponding the extra term to $K_1(R_3)$.

\end{ex}

\begin{rem}
Observe that the proposed method permits computations of the Bredon homology (with respect to the family of virtually cyclic groups) of $A_n$ with respect to the $K$-theory of any ring for which there is some knowledge of the algebraic $K$-groups.
We remark that the computation of the lower algebraic $K$-theory groups is a hot topic nowadays (see for example \cite{LaOr07}, \cite{HiJu2021} or \cite{GJM18}), so it seems possible that good knowledge about the $E^2$-term of the Atiyah-Hirzebruch spectral sequence is achieved in these cases, even for more general families of Artin groups.

\end{rem}



When the ring $R$ is regular, it is possible to take a shortcut in order to compute the left-hand side of the Farrell-Jones conjecture. Observe that Theorem 0.1 in \cite{LS16}, that establishes a splitting $$H^G(\E G,\mathbf{K} (R))\simeq H^G(\underline{E} G,\mathbf{K} (R))\oplus H^G(\E G, \underline{E} G, \mathbf{K} (R)),$$ for every group $G$ and ring $R$.
Now if $R$ is regular and $G$ is torsion-free, the second term of the direct sum vanishes \cite[Proposition 2.6]{LR05}. Likewise, when $G$ is torsion-free $\underline{E}  G=EG$, the classical universal space for principal $G$-bundles, and then $H^G(\underline{E}  G,\mathbf{K}(R))=H(BG, \mathbf{K}(R))$, being the latter ordinary homology. Hence, in the case $G=A_n$, it is enough to compute $H^(BA_n, \mathbf{K}(R))$ to obtain the left-hand side of the Farrell-Jones conjecture for these groups. When the coefficients $\mathbf{K}(R)$ are known, and taking into account that the Artin groups of dihedral type are one-relator, Lemma 16.21 in \cite{Luc21} provides an accurate description of these homology groups, corresponding case i) of the lemma to $n$ even and case ii) to $n$ odd. For example, when $K=\mathbb{F}_p$ for $p$ prime, the classical results of \cite{Qui73} provide a complete knowledge of the groups $H^(BA_n, \mathbf{K}(R))$ , while for $K=\mathbb{Z}$ much information is available (\cite[page 2]{Wei05}). When $R$ is not regular, however, this strategy does not work, as in this case groups in the left-hand side of Farrell-Jones. In this context, we expect that our results on Bredon homology of $A_n$ (znd in particular the last examples of this section) can bring some light over the groups $H^{A_n}(\E A_n,\mathbf{K} (R))$, via appropriate computations in the equivariant Atiyah-Hirzebruch spectral sequence. Observe that, according to Theorem \ref{Thm:Bredon}, this spectral sequence has four columns. This fact indicates that the sequence should collapse at most in the page $E_4$, and hence this page should provide the knowledge of Farrell-Jones groups. The analysis of the differentials, as well as the subsequent description of the $E_3$ and $E_4$-pages, seems a difficult and interesting future line of research.

\noindent{\textbf{{Acknowledgments.}}}

We warmly thank D. Juan-Pineda, W. L\"{u}ck, N. Petrosyan, L.J. S\'anchez-Salda\~{n}a, V. Srinivasan and J. Stienstra and an anonymous referee for their useful comments.


\end{document}